\documentclass[11pt,reqno,bold]{paper}
\usepackage{geometry}
\usepackage{amsmath}
\usepackage{amssymb}
\usepackage{amsfonts}
\usepackage{mathrsfs}
\usepackage{dsfont}
\usepackage{enumerate}
\usepackage{comment}
\usepackage{amsthm}
\usepackage{color}
\definecolor{darkblue}{rgb}{0.0,0.0,0.4}
\usepackage[pdfauthor={Al Balushi, Tsogtgerel},pdftitle={Adaptive finite element method for the biharmonic problem},pdfsubject={adaptive methods, convergence rate, biharmonic problem},colorlinks,citecolor=darkblue,linkcolor=darkblue,urlcolor=darkblue]{hyperref}

\newtheorem{theorem}{Theorem}[section]

\newtheorem{lemma}[theorem]{Lemma}
\newtheorem{corollary}[theorem]{Corollary}
\theoremstyle{definition}
\theoremstyle{remark}
\newtheorem{remark}[theorem]{Remark}

\usepackage{algorithmicx}
\usepackage[ruled]{algorithm}
\usepackage{algpseudocode}

\excludecomment{comment:quasi}

\newcommand{\ms}[1]{{\mathbf{#1}}}

\newcommand{\ds}[1]{{\mathds{#1}}}
\newcommand{\bs}[1]{{\boldsymbol{#1}}}
\renewcommand{\rm}[1]{{\mathrm{#1}}}
\newcommand{\bb}[1]{{\mathbb{#1}}}
\renewcommand{\cal}[1]{{\mathcal{#1}}}
\newcommand{\scr}[1]{{\mathscr{#1}}}
\newcommand{\n}{\ms{n}}
\newcommand{\diff}{\partial}
\newcommand{\grad}{\nabla}
\newcommand{\Lap}{\Delta}

\renewcommand{\dfrac}[2]{\frac{\diff #1}{\diff #2}}
\newcommand{\semi}[2]{{|#1|}_{#2}}
\newcommand{\norm}[2]{\|#1\|_{#2}}
\newcommand{\inner}[1]{\left\langle#1\right\rangle}
\newcommand{\enorm}[2]{|\!|\!|#1|\!|\!|_{#2}}
\newcommand{\jump}[2]{\bb{J}_{#2}\!\left(#1\right)}

\newcommand{\supp}{\mathrm{supp}\,}
\newcommand{\diam}{\rm{diam}\,}
\newcommand{\lapprox}{\preceq}

\newcommand{\wh}[1]{\widehat{#1}}
\newcommand{\eps}{\varepsilon}
\renewcommand{\u}{u}
\renewcommand{\v}{v}
\newcommand{\w}{w}
\newcommand{\e}{e}
\newcommand{\U}{U}
\newcommand{\V}{V}
\newcommand{\W}{W}
\newcommand{\X}{\bb{X}}
\renewcommand{\a}{a}

\newcommand{\marked}{\scr{M}}
\newcommand{\face}{\tau}
\newcommand{\cell}{\tau}
\newcommand{\edge}{\sigma}

\newcommand{\diffe}{\diff_\edge}
\newcommand{\eff}{f}

\newcommand{\est}{\eta}
\newcommand{\estP}{\eta_\mesh}
\newcommand{\up}{\U}
\newcommand{\upp}{\U_\ast}

\newcommand{\Ip}{I_\mesh}

\newcommand{\Ho}{H_0^2}

\newcommand{\mesh}{P}
\newcommand{\bedges}{\cal{G}}
\newcommand{\edges}{\cal{E}}
\newcommand{\Xp}{\bb{X}_\mesh}
\newcommand{\Xpp}{\bb{X}_{\mesh_\ast}}
\newcommand{\ap}{a_\mesh}

\newcommand{\Op}{\cal{L}}

\newcommand{\ellf}{\ell_f}
\newcommand{\osc}{\rm{osc}}

\newcommand{\Ccoer}{C_\rm{coer}}
\newcommand{\Ccont}{C_\rm{cont}}

\newcommand{\cshape}{c_\rm{shape}}
\newcommand{\Crel}{C_\rm{rel}}
\newcommand{\Cdrel}{C_\rm{dRel}}
\newcommand{\Ceff}{C_\rm{eff}}
\newcommand{\Cest}{C_\rm{est}}
\newcommand{\qest}{q_\rm{est}}
\newcommand{\Clip}{C_\rm{lip}}

\newcommand{\ctrace}{d_0}
\newcommand{\cinv}{d_1}
\newcommand{\cdtrace}{d_2}

\newcommand{\cc}{d_4}
\newcommand{\cproj}{c_1}

\title{A convergent boundary-condition conforming adaptive spline-based finite element method for the bi-Laplace operator}
\author{Ibrahim Al Balushi}
\institution{McGill University}
\date{\today}

\begin{document}
\maketitle
\begin{abstract}
We establish the convergence of an adaptive spline-based finite element method of a fourth order elliptic problem.
\end{abstract}
\section{Introduction}
Design of optimal meshes for finite element analysis is a topic of extensive research going back to the early seventies.
Among the rich variety of strategies explored, the first mathematical framework for automatic optimal mesh generation was laid in the seminal work of Babushka and Rheinboldt \cite{babuvvska1978error}.
They introduce a class of computable a posteriori error estimates for a general class of variational problems which provides a strategy of extracting localized approximations of the numerical error of the exact solution.
The derived computable estimates are shown to form an upper bound with the numerical error, justifying the validity of the estimator, and a lower bounds which ensures efficient refinement; i.e, refinement where only necessar.
This established equivablence with the numerical error eluded to promising potential of practicality and robustness.
The theoretical results led to a heuristic characterization of optimal meshes through the even distribution a posteriori error quantities over all mesh elements, providing a blue-print for adaptive mesh generation.
The first detailed discription and performance analysis for a simple and accessible a posteriori error estimator was conceived for a one-dimensional elliptic and parabolic second-order Poisson-type problems in \cite{babuvska1978posteriori}.
The analysis was significantly improved in \cite{verfurth1994posteriori} for two-dimensional scenarios and develops numerous techniques used in the derivation of a posteriori estimates untill today.\\
\\
The first convergence was given in \cite{bubuvska1984feedback} in one-dimensions and later extended to two-dimenions by Dorfler \cite{dorfler1996convergent}.
Combing the advent of a novel \emph{marking} strategy and \emph{finess} assumptions of the initial mesh, \cite{dorfler1996convergent}.
Element marking directs the refinement procedure to select user-specified ratio of elements with highest error indicators relative to the total estimation.
The proposed strategy is later shown to be optimal \cite{cascon2008quasi}.
The initial mesh assumption was placed to ensure problem datum, such as source function and boundary values, are sufficiently resolved for detection by the solver.
The aim is to ensure error reduction of the estimator and thus monotone convergence of numerical error is achieved through contraction of consecutive errors in energy norm at every step.
but at the expense of potential over-refinement of the initial mesh. This was achieved using a local counter part of the efficiency estimate described above.\\
\\
Morin et al \cite{morin2000data} came to the realization that the averaging of the data has an unavoidable interference with the estimator error reduction irrespective of quadrature and was due to the avergaing of finer feathures of data brought by finite-dimensional approximations.
This averaging was quantified into an \emph{oscillation term}, a quantity tightly related to the criterion used in the initial mesh assumption used in \cite{dorfler1996convergent} but it provided a sharper representation of the underlying issue.
As a result the initial mesh assumption was removed in \cite{morin2000data} and replaced with a Dolfer-type marking criterion, \emph{separate marking}, for the data oscillation.
Unfortunately, the relaxation of a fine initial mesh had the unintended consequence of losing the strict monotone behaviour of numerical error decay.
It led to the introduction of the \emph{interior node property} to ensure error reduction with every step so as to ensure two consequence solutions will not be the same unless they are equal to the exact one; but at the expense of introducing over refinement. Each marked elementin in a two-dimensional triangular mesh undergoes three bisections ensuring an interior node, which furnishes us with a local lower bound and thus recovers strict error reduction with every iteration.
The results were extended to saddle-problems in \cite{morin2002convergence} and genarlized into abstract Hilbert setting in \cite{morin2008basic}.\\
\\
For the better part of the 2000's Morin-Nochetto-Sierbert algorithm (MNS) champoined adaptive finite element methods AFEM of linear elliptic problems after which the analysis was refined by Cascon \cite{cascon2008quasi} in concrete setting where they did not rely on a local lower bound for convergence which led to the ultimate removal of the costly interior node property and separate marking for oscillation. This was done while achieving quasi-optimal mesh complexity; see \cite{binev2004adaptive},\cite{stevenson2005optimal} for dertails. It was realized that strict reduction of error in energy norm cannot be gaurenteed whenever consequetive numerical solutions coinside but strict monotone decay is be obtain with respect to a suitable \emph{quasi-norm}. The result was extended to abstract Hilbert by Siebert \cite{siebert2010converg} which is now widely considered state of the art analysis of AFEM among the adaptive community and It hinges on the following ingredients: a global upper bound justifying the validity of the a posteriori  estimator, a Lipschitz property of the estimator as a function on the discrete finite element trail space indicating suitable sensitivity in variation within the trial space, a Pythgurous-type relation furnished by the variational and discrete forms and any suitable making strategy akin to that of Dofler; one that aims to equally distributes the elemental error estimates.\\
\\

\section{Problem set up and Adaptive method}
Let $\Omega$ be a bounded domain in $\bb{R}^2$ with polygonal boundary $\Gamma$.
For a source function $f\in L^2(\Omega)$ we consider the following homogenous Dirichlet boundary-valued problem
\begin{eqnarray}\label{eq:pde}
\Op\u(x):=\Lap^2\u(x)=f(x)&&\text{in}\ \Omega\\
u=\diff u/\diff\nu=0&& \text{on}\ \Gamma.\nonumber
\end{eqnarray}
The adaptive procedure iterates over the following modules
\begin{equation}\label{eq:afem}
\boxed{\ms{SOLVE}}\longrightarrow\boxed{\ms{ESTIMATE}}\longrightarrow\boxed{\ms{MARK}}\longrightarrow\boxed{\ms{REFINE}}
\end{equation}
The module $\textbf{SOLVE}$ computes a hierarchical polynomial B-spline approximation $\U$ of the solution $\u$ with respect to a hierarchical partition $\mesh$ of $\Omega$.
A detailed discussion on the nature of such partitions will be carried in Section~\ref{sec:spline}
For the module \textbf{ESTIMATE}, we use a residual-based error estimator $\eta_\mesh$ derived from the a posteriori analysis in Section~\ref{sec:post}.
The module \textbf{MARK} follows the D\"olfer marking criterion of \cite{dorfler1996convergent}.
Finally, the module \textbf{REFINE} produces a new refined partition $\mesh_\ast$ satisfying certain geometric constraints described in Section~\ref{sec:spline} to ensure sharp local approximation.\\
\\
\subsection{Notation}
We begin by laying out the notational conventions and function space definitions used in this presentation.
Let $\mesh$ be a partition of domain $\Omega$ consisting of square cells $\cell$ following the structure described in \cite{vuong2011hierarchical},\cite{al2018adaptivity}.
Denote the collection of all interior edges of cells $\cell\in\mesh$ by $\edges_\mesh$ and all those along the boundary $\Gamma$ are to be collected in $\bedges_\mesh$.
We assume that cells $\tau$ are open sets in $\Omega$ and that edges $\sigma$ do not contain the vertices of its affiliating cell.
Let $\diam(\omega)$ be the longest length within a Euclidian object $\omega$ and set $h_\face:=\diam(\face)$ and $h_\edge:=\diam(\edge)$. Then let the mesh-size $h_\mesh:=\max_{\face\in\mesh}h_\face$.
Define the boundary mesh-size function $h_\Gamma\in L^\infty(\Gamma)$ by
\begin{equation}
h_\Gamma(x)=\sum_{\sigma\in\bedges_\mesh}h_\sigma\ds{1}_\sigma(x),
\end{equation}
where the $\ds{1}_\sigma$ are the indicator functions on boundary edges.
Let $H^s(\Omega)$, $s>0$, be the fractional order Sobolev space equipped with the usual norm $\norm{\cdot}{H^s(\Omega)}$; see references \cite{adams1975sobolev},\cite{grisvard2011elliptic}.
Let $H^s_0(\Omega)$ be given as the closure of the test functions $C_c^\infty(\Omega)$ in $\norm{\cdot}{H^s(\Omega)}$.
The semi-norm $\semi{\cdot}{H^s(\Omega)}$ defines a full norm on $H^s_0(\Omega)$ by virtue of Poincar\'e's inequality. Moreover, the semi-norm $\norm{\Lap\cdot}{L^2(\Omega)}$ defines a norm on $H^2_0(\Omega)$.
%
%
By $H^{-2}(\Omega)=(H_0^{2}(\Omega))'$ the dual of $H^{2}(\Omega)$ with the induced norm
\begin{equation}
\norm{F}{H^{-2}(\Omega)}=\sup_{\v\in H_0^2(\Omega)}\frac{\inner{F,\v}}{\norm{\v}{H^2(\Omega)}}.
\end{equation}
\subsection{Weak formulation}
The natural weak formulation to the PDE \eqref{eq:pde} reads
\begin{equation}\label{eq:cwp}
\text{Find}\ \u\in\Ho(\Omega)\ \text{such that}\ \a(\u,\v)=\ellf(\v)\ \text{for all}\ \v\in\Ho(\Omega),
\end{equation}
where $\a:\Ho(\Omega)\times\Ho(\Omega)\to\bb{R}$ is be the bilinear form $\a(\u,\v)=(\Lap\u,\Lap\v)_{L^2(\Omega)}$ and $\ellf(\v)=(f,\v)_{L^2(\Omega)}$.
The energy norm $\enorm{\cdot}{}:=\sqrt{\a(\cdot,\cdot)}\equiv\norm{\Lap\cdot}{L^2(\Omega)}$ is one for which the form $\a$ is continuous and coercive on $H^2_0(\Omega)$, with unit proportionality constants. The existence of a unique solution is therefore ensured by Babuska-Lax-Milgram theorem.
The variational formulation \eqref{eq:cwp} is consistent with the PDE \eqref{eq:pde} under sufficient regularity considerations; if $\u\in H^4(\Omega)\cap H^2_0(\Omega)$ satisfies \eqref{eq:cwp} then $\u$ satisfies \eqref{eq:pde} in the classical sense by virtue of the Du Bois-Reymond lemma.

\subsection{Spline spaces and hierarchical partitions}\label{sec:spline}
We consider a hierarchical polynomial spline space as the discrete trial and test space.
For completeness we describe the construction.
Let $\cal{S}^0\subset \cal{S}^{1}\subset\cdots\subset \cal{S}^{L-1}$ be a hierarchy of $L$ tensor-product multivariate spline spaces defined on $\Omega$.
For each hierarchy level $\ell$, we obtain $B$-spline polynomial basis $\cal{B}^\ell$ of degree $r\ge2$ defined on a tensor-product mesh $G^\ell$ partitioning of $\Omega$ generated by tensorizing translations and dilations of an $r$-th degree cardinal B-spline $b^r$ defined by recursive convolution with the characteristic function:
\begin{equation}
b^k(x)=\int_0^{k+1}b^{k-1}(x-t)\ds{1}_{[0,1)}(t)\,dt,\quad b^0:=\ds{1}_{[0,1)}(x).
\end{equation}
Partition $G^{\ell+1}$ is obtained from $G^{\ell}$ via uniform dyadic subdivisions which will insure the nesting $\cal{S}^{\ell}\subset\cal{S}^{\ell+1}$. That is, if $s\in\cal{S}^\ell$ then we may express $s$ in terms of $\cal{B}^{\ell+1}$:
\begin{equation}\label{eq:nestedrelation}
s=\sum_{\beta\in\cal{B}^{\ell+1}}c_\beta^{\ell+1}(s)\beta,
\end{equation}
where $c_\beta^{\ell+1}(s)$ are the coefficients of $s$ when expressed in $\bb
{B}^{\ell+1}$.
From classical spline theory, it is well-known that B-splines are locally linearly independent, they are non-negative, they are supported locally and form a partition of unity.
We are now in a position to define hierarchical mesh configuration.
A cell $\cell$ of level $\ell$ is said to \emph{active} if $\cell\in G^\ell$ and $\cell\cap\Omega^{\ell+1}=\emptyset$. 
A subdomain $\Omega^\ell$ of $\Omega$ is defined as the closure of the union of active Cells $\cell$ in $G^\ell$. 
With subdomain hierarchy $\bs{\Omega}^L=\{\Omega^\ell\}_{\ell=0}^{L-1}$ of closed domains ${\Omega}^0\supseteq{\Omega}^1\supseteq\cdots\supseteq{\Omega}^{L-1}$, with $\rm{int}({\Omega}^0)={\Omega}$ and ${\Omega}^{L}=\emptyset$, we define a hierarchical $\mesh$ partitioning of $\Omega$ as a mesh satisfying the following conditions:
\begin{enumerate}
\item 
Members of $\mesh$ are active cells from $G^\ell$, $0\leq\ell\leq L-1$.
\item 
All cells $\tau$ in $\mesh$ are disjoint.
\item 
The interior of the closure of the union $\cup\{\cell:\cell\in\mesh\}$ is equal to $\Omega$.
\end{enumerate}
A \emph{Hierarchical B-spline} (HB-spline) basis $\cal{H}$ with respect to hierachical partition $\mesh$ is defined as
\[
\cal{H}_\mesh=\left\{\beta\in\cal{B}^{\ell}:\supp\beta\subseteq\wh{\Omega}^\ell\ \wedge\ \supp\beta\not\subseteq\wh{\Omega}^{\ell+1}\right\}.
\]
A recursive definition is given in \cite{speleers2016effortless}.
A basis function $\beta$ of level $\ell$ is said to \emph{active} if $\beta\in\cal{B}^\ell\cap\cal{H}$, otherwise it is \emph{passive}.
The basis $\cal{H}_\mesh$ inherits much of the key properties of tensor-product B-spline bases: they are locally linearly independent, they are non-negative and they have local support \cite{vuong2011hierarchical}. However, the basis does not form a partition of unity which could pose a problem to approximation stability.
It is possible to modify $\cal{H}_\mesh$ into forming a partition of unity through scaling \cite{vuong2011hierarchical} but instead we use a truncation procedure utilizing the relation \eqref{eq:nestedrelation} to produce a new basis that recovers the partition of unity while preseving all the desirable properties of $\cal{H}_\mesh$.
We define a truncation operatoe of a spline function $s\in\cal{S}^\ell$:
\begin{equation}\label{eq:truncation}
\rm{trunc}^{\ell+1}s:=\sum_{\beta\in\cal{B}^{\ell+1}:\supp\beta\not\subseteq{\Omega}^{\ell+1}}c_\beta^{\ell+1}(s)\beta.
\end{equation}
In simple terms, the truncation removes contributions coming from \emph{active} basis functions in $\cal{B}^{\ell+1}$ thus reducing the support $s$ from reaching too far into $\Omega^{\ell+1}$.
By recursive application of \eqref{eq:truncation} to each spline $\beta\in\cal{H}_\mesh$:
\begin{equation}
\rm{Trunc}^{\ell+1}\beta:=\rm{trunc}^{L-1}(\rm{trunc}^{L-2}(\cdots(\rm{trunc}^{\ell+1}\beta)))
\end{equation}
we obtain a modified hierarchical B-spline basis, a truncated hierarchical B-spline (THB-spline) basis $\cal{T}_\mesh$ with respect to partition $\mesh$:
\begin{equation}
\cal{T}_\mesh=\left\{\rm{Trunc}^{\ell+1}\beta:\beta\in\cal{B}^\ell\cap\cal{H}_\mesh,\ \ell=0:L-1\right\}
\end{equation}
The basis $\cal{T}_\mesh$ retains all of the aforementioned properties of its hierarchical counterpart $\cal{H}_\mesh$ while forming a partition of unity. 

\subsection{Admissible partitions}
For local and stable approximation we need to control the infleuence of each basis function. With additional restrictions on the structure of partitions $\mesh$ we can guarantee that the number of basis functions acting on any point is bounded and that the diameter of the support of a basis function is comparable to any cell in its support.
A partition $\mesh$ is said to be \emph{admissible} if the truncated basis functions in $\cal{T}$ which has support on $\tau\in\mesh$ belong to at most two levels successive levels.
The support extension of a cell $\cell\in\cal{G}^\ell$ with respect to level $k\leq\ell$ is defined as
\begin{equation}
S(\cell,k):=\left\{\cell'\in G^k:\exists\beta\in\cal{B}^k\ s.t\ \supp\beta\cap\cell'\neq\emptyset\ \wedge\ \supp\beta\cap\cell\neq\emptyset\right\}
\end{equation}
Note that the support extension consist of cells from the tensor-product mesh $\cal{G}^k$.
To assess the locality of the basis; i.e, the influence of basis functions have on active cells, it is useful to consider a support extension consisting of all active cells belonging to its support regardless of level. For $\tau\in\mesh$ define
\begin{equation}
\omega_\tau=\bigcup_{\ell=0}^{L-1}S(\tau,\ell)\cap\mesh
\equiv\{\tau'\in\mesh:\supp\beta\cap \tau'\neq\emptyset\implies\supp\beta\cap \tau\neq\emptyset\},
\end{equation}
indicating the collection of all supports for basis function $\beta$'s whose supports intersect $\tau$.
Analogously, we denote the support extension for an edge $\sigma\in\edges_\mesh\cup\bedges_\mesh$ by
\begin{equation}
\omega_\sigma=\{\tau\in\mesh:\supp\beta\cap\tau\neq\emptyset\implies\supp\beta\cap \tau\neq\emptyset,\ \sigma\subset\diff\tau\}.
\end{equation}
The following auxiliary subdomain provides a way to ensure mesh admissibility 
\begin{equation}
\cal{U}^\ell:=\bigcup\left\{\overline{\cell}:\cell\in G^\ell\ \wedge\ S(\cell,\ell)\subseteq\Omega^\ell\right\}
\end{equation}
\begin{lemma}
Let $\bs{\Omega}^L$ be a subdomain hierarchy with respect to partition $\mesh$ of domain $\Omega$. If 
\begin{equation}
\Omega^\ell\subseteq\cal{U}^{\ell-1}
\end{equation}
for $\ell=2:L-1$, then $\mesh$ is an admissible partition.
\end{lemma}
\begin{proof}
See \cite{buffa2016adaptive}.
\end{proof}
In other words, $\cal{U}^\ell$ represents the biggest subset of $\Omega^\ell$ so that the set of B-splines in $\cal{B}^\ell$ whose support is contained in $\Omega^\ell$ spans the restriction of $\cal{S}^\ell$ to $\cal{U}^\ell$. 

\subsection{The adaptive method}

We now discuss the modules \textbf{SOLVE}, \textbf{ESTIMATE},
\textbf{MARK} and \textbf{REFINE} in detail.
\subsubsection*{The module SOLVE}
The space of piecewise polynomials of degree $r\ge2$ defined on a partition $\mesh$ will be given by
\begin{equation}
\cal{P}^r_\mesh(\Omega)=\prod_{\tau\in\mesh}\bb{P}_r(\tau).
\end{equation}
Assuming we have at our disposal a polynomial B-spline space $\Xp\subset\cal{P}_\mesh^r(\Omega)\cap H^2_0(\Omega)$ then the discrete problem reads
\begin{equation}\label{eq:dcp}
\U=\ms{SOLVE}[\mesh,f]:\quad\text{Find}\ \U\in\Xp\ \text{such that}\ \a(\U,\V)=\ellf(\V)\ \text{for all}\ \V\in\Xp.
\end{equation}
%
The linear system is numerically stable and consistent with \eqref{eq:cwp} in the sense that $\a(\u,\V)=\ellf(\V)$ for every $V\in\Xp$ and therefore we are provided with Galerkin orthogonality:
\begin{equation}\label{eq:cgo}
\a(\u-\U,\V)=0\quad\forall\V\in\Xp.
\end{equation}
Moreover, the spline solution will serve as an optimal approximation to $\u$ in $\Xp$ with respect to $\enorm{\cdot}{}$. Indeed, we have for any $\V\in\Xp$,
\begin{equation}
\enorm{\u-\U}{}^2\leq\Ccoer^{-1}\ap(\u-\U,\u-\V)\leq\frac{\Ccont}{\Ccoer}\enorm{\u-\U}{}\enorm{\u-\V}{}
\end{equation}
\subsubsection*{The module ESTIMATE}
For a continuous function $\v$ we define the jump operator across interface $\sigma$.
\begin{equation}
\jump{\v}{\sigma}=\lim_{t\to0}[\v(x+t\sigma)-\v(x-tx)],\quad x\in\sigma.
\end{equation}
The adaptive refinement procedure of method \eqref{eq:afem} will aim to reduce the error estimations instructed by the cell-wise error indicators:
\begin{equation}\label{eq:indicator}
\eta_\mesh^2(\V,\face)=h_\face^4\norm{f-\Op\V}{L^2(\tau)}^2
+\sum_{\sigma\subset\diff\tau}
\left(\textstyle h_\sigma^{3}\left\norm{\jump{\dfrac{\Lap\V}{\n_\sigma}}{\sigma}\right}{L^2(\sigma)}^2
+h_\sigma\norm{\jump{\Lap\V}{\sigma}}{L^2(\sigma)}^2\right).
\end{equation}
We can define the indicators on subsets of $\Omega$ via:
\begin{equation}\label{eq:estimator}
\eta_\mesh^2(\V,\omega)=\sum_{\tau\in \mesh:\tau\subset\omega}\eta_\mesh^2(\V,\face),\quad\omega\subseteq\Omega
\end{equation}
To each cell $\tau$ in mesh $\mesh$ the error indicators \eqref{eq:indicator} will assign error estimations:
\begin{equation}
\{\eta_\tau:\tau\in\mesh\}=\ms{ESTIMATE}[\U,\mesh]:\quad\eta_\tau:=\eta_\mesh(\U,\tau)
\end{equation}
\subsubsection*{The module MARK}
We follow the Dorlfer marking strategy \cite{dorfler1996convergent}: For $0<\theta\leq1$,
\begin{equation}\label{markingstrategy}
\text{Find minimal spline set}\ \marked:\quad\sum_{\cell\in\marked}\eta^2_\mesh(\U,\cell)\ge\theta\sum_{\tau\in\mesh}\eta_\mesh^2(\U,\cell).
\end{equation}
To ensure minimal cardinality of $\cal{M}$ in the marking strategy one typically undergoes QuickSort which has an average complexity of $\cal{O}(n\log n)$ to produce the indexing set $J$.

\subsubsection*{The module REFINE}
The refinement framework is designed to preserve the structure described in the previous section hinges on extending the marked cells obtained from module \textbf{MARK} to a set $\omega_{R_{\mesh\to\mesh_\ast}}$ for which the new mesh $\mesh_\ast$ is admissible. 
We define the \emph{neighbourhood} of $\cell\in\mesh\cap\cal{G}^{\ell}$ as
\[
\cal{N}(\mesh,\cell):=\left\{\cell'\in\mesh\cap G^{\ell-1}:\exists\cell''\in S(\cell,\ell),\ \cell''\subseteq\cell'\right\}
\]
when $\ell-1>0$, and $\cal{N}(\mesh,\cell)=\emptyset$ otherwise. To put in concrete terms, the neighbourhood $\cal{N}(\mesh,\cell)$ of an active cell in $G^\ell$ consist of active cells $\cell'$ of level $\ell-1$ overlapping the support extension of $\cell$ with respect to level $\ell$.
\begin {algorithm}[h]
\small
\caption {Recursive refinement $\mesh_\ast\leftarrow\ms{recursive\_refine}\,[\mesh,\cell]$}
\begin {algorithmic}[1]
\For {all $\cell'\in\cal{N}(\mesh,\cell)$}
\State $\mesh\leftarrow\ms{recursive\_refine}\,[\mesh,\cell']$
\EndFor
\If {$\cell\in\mesh$}
\State $\{\cell_j\}_{j=1}^{4}\leftarrow$ dyadic-refine $\cell$
\State $\mesh_\ast\leftarrow(\mesh\backslash\cell)\cup\{\cell_j\}_{j=1}^{4}$
\EndIf
\end {algorithmic}
\end {algorithm}
\begin{algorithm}[h]
\small
\caption{Carry admissible mesh refinement $\mesh_\ast\leftarrow\ms{mesh\_refine}\,[\mesh,\scr{M},m]$}
\begin{algorithmic}[1]
\For {$\cell\in\scr{M}$}
\State $\mesh\leftarrow\ms{recursive\_refine}\,[\mesh,\cell]$
\EndFor
\State $\mesh_\ast\leftarrow\mesh$
\end{algorithmic}
\end{algorithm}
Procedure \textbf{REFINE} will ensure that for a constant $\cshape>0$, depending only on the polynomial degree of the spline space, all considered partitions therefore will satisfy the shape-regularity constraints:
\begin{eqnarray}\label{eq:sr}
\nonumber\sup_{\mesh\in\scr{P}}\max_{\tau\in\mesh}\#\left\{\tau\in\mesh:\tau\in\omega_\tau\right\}\leq\cshape&&\text{(finite-intersection property)},\\
\sup_{\mesh\in\scr{P}}\max_{\tau\in\mesh}\frac{\diam(\omega_\tau)}{h_\tau}\leq \cshape
&&(\text{graded}).
\end{eqnarray}
For any two partitions $\mesh_1,\mesh_2\in\scr{P}$ there exists a common admissible partition in $\scr{P}$, called the \emph{overlay} and denoted by $\mesh_1\oplus \mesh_2$, such that 
\begin{equation}\label{eq:MeshOverelay}
\#(\mesh_1\oplus\mesh_2)\leq\#\mesh_1+\#\mesh_2-\#\mesh_0.
\end{equation}
Moreover, shown in \cite{buffa2016complexity},
if the sequence $\{\mesh_\ell\}_{\ell\ge1}$ is obtained by repeating the step $\mesh_{\ell+1}:=\ms{ REFINE}\,[\mesh_\ell,\scr{M}_\ell]$ with $\scr{M}_\ell$ any subset of $\mesh_\ell$, then for $k\ge1$ we have that
\begin{equation}\label{eq:MarkingComplexity}
\#P_k-\#P_\ell\leq \Lambda\sum_{\ell=1}^{k}\#\scr{M}_\ell.
\end{equation}
where $\Lambda>0$ which will depend on the polynomial degree $r$.

%
\section{A posteriori estimates}\label{sec:post}	
We define the residual quantity $\scr{R}\in H^{-2}(\Omega)$ by
\begin{equation}\label{eq:residual}
\inner{\scr{R},\v}=\a(\u-\U,\v),\quad\v\in H_0^2(\Omega).
\end{equation}
In view of continuity and coercivity of the bilinear form we readily have sharp a posteriori estimates for $\u-\U$
\begin{equation}
\Ccont^{-1}\norm{\scr{R}}{H^{-2}(\Omega)}\leq\enorm{\u-\U}{}\leq\Ccoer^{-1}\norm{\scr{R}}{H^{-2}(\Omega)}.
\end{equation}
The quantity $\norm{\scr{R}}{H^{-2}(\Omega)}$ is computable since it only depends on available discrete approximation of solution $\u$. We follow the techniques devised in \cite{verfurth1994posteriori},\cite{ainsworth2011posteriori} to approximate $\norm{\scr{R}}{H^{-2}(\Omega)}$.
\subsection{Approximation in $\Xp$}
To quantify the approximation power of $\Xp$, we use a quasi-interpolant $\Ip:L^2(\Omega)\to \Xp$;
see~\cite{speleers2017hierarchical},\cite{speleers2016effortless},\cite{bazilevs2006isogeometric} for the detailed construction of $\Ip$. The following theorem summarizes the local approximation properties of $\Ip$.
Various spline-based quasi-interpolants have been studied extensively and amounts to choosing dual-functionals $\lambda$. A suitable choice for B-spline basis is \cite{bazilevs2006isogeometric}.  To each level $\ell$ we assume we have in hand 
\begin{equation}\label{eq:LevelWiseInterpolants}
I^\ell(v)=\sum_{\beta\in\cal{B}^\ell}\lambda_\beta(v)\beta,\quad\v\in L^2(\Omega)
\end{equation}
such that $I^\ell(s)=s$ for every $s\in\cal{S}^\ell$. In \cite{speleers2016effortless} it is shown that it is sufficient to define $\Ip$ with
\begin{equation}
\Ip(v)=\sum_{\ell=0}^{L-1}\sum_{\beta\in\cal{B}^\ell\cap\cal{H}_\mesh}\lambda_\beta(v)\rm{Trunc}^{\ell+1}\beta,\quad\v\in L^2(\Omega)
\end{equation}
with each $\lambda_\beta$ being that of the one in the level-wise interpolant \eqref{eq:LevelWiseInterpolants}. 
\begin{theorem}[Quasi-interpolation]\label{thm:quasiinterpolant}
	There exists a quasi-interpolation projection operator $\Ip:L^2(\Omega)\to\Xp$ such that, for a constant $ \cshape>0$, independent of the refinement, and $0\leq t\leq2$,   
	\begin{equation}
	\norm{\Ip\v}{L^2(\tau)}\leq c_\mathrm{shape}\norm{\v}{L^2(\omega_\tau)}\quad \text{for}\ \tau\in\mesh ~\text{and}~\v\in L^2(\omega_\tau)
	\end{equation}
	and the approximation properties
	\begin{equation}
	\forall\v\in H^2(\omega_\tau),\quad|\v-\Ip\v|_{H^t({\tau})}\leq c_\mathrm{shape} h_{\tau}^{2-t}|\v|_{H^2(\omega_{\tau} )}\quad\forall\v\in H^2(\omega_\tau)
	\end{equation}
	and
	\begin{equation}
	\forall\sigma\in\cal{E}_\mesh\cup\cal{G}_\mesh,\quad |\v-\Ip\v|_{H^t(\sigma)}\leq c_\mathrm{shape} h_S^{3/2-t}|\v|_{H^2(\omega_\sigma)}\quad\forall\v\in H^2(\omega_\sigma) 
	\end{equation}
\end{theorem}

Recall the general trace theorem \cite{adams1975sobolev},\cite{grisvard2011elliptic} for cells $\tau\in\mesh$ and edges $\sigma\in\bedges_\mesh$ with $\sigma\subset\diff\tau$.
For a constant $\ctrace>0$
\begin{equation}\label{eq:GeneralTrace}
\norm{\v}{L^2(\sigma)}^2\leq \ctrace\left(h_\sigma^{-1}\norm{\v}{L^2(\tau)}^2+h_\sigma\norm{\grad\v}{L^2(\tau)}^2\right)\quad\forall\v\in H^1(\Omega).
\end{equation}
\begin{lemma}[Auxiliary discrete estimate]\label{lem:ie}
Let $\face\in\mesh$.
Then for $\cinv>0$, depending only on polynomial degree $r$, for $0\leq s\leq t\leq r+1$ we have
\begin{equation}\label{eq:inve:lem:ie}
\semi{\V}{H^t(\tau)}\leq\cinv h_\tau^{s-t}\semi{\V}{H^s(\tau)}\quad\forall\V\in\bb{P}_r(\tau),
\end{equation}
and if $\edge\subset\diff\face$, for a constant $\cdtrace>0$ we have
\begin{equation}\label{eq:dtrace:lem:ie}
\norm{\V}{L^2(\sigma)}\leq\cdtrace h_\sigma^{-1/2}\norm{\V}{L^2(\tau)}\quad\forall\V\in\bb{P}_r(\tau),
\end{equation}
where $\cdtrace:=\ctrace\max\{1,\cinv\}$.
\end{lemma}
\begin{remark}
The constants $\cinv,\ \ctrace,\ \cdtrace$ all depend on the polynomial degree and the reference cell or edge; $\hat{\tau}=[0,1]^2$ or $\hat{\sigma}=[0,1]$.
From now, for a simpler presentation of the analysis, we combined all these constants, and their powers into a unifying constant $c_\ast$
\end{remark}

\subsection{The global upper bound}
We prove that the proposed error estimator is reliable.
\begin{lemma}[Estimator reliability]\label{lem:er}
Let $\mesh$ be a partition of $\Omega$ satisfying Conditions \eqref{eq:sr}.
The module $\ms{ESTIMATE}$ produces a posteriori error estimate $\eta_\mesh$ for the discrete error such that for a constants $C_\rm{rel}>0$,
\begin{equation}\label{eq:res:lem:er}
\begin{split}
\enorm{\u-\U}{}^2&\leq C_\rm{rel}\eta_\mesh^2(\U,\Omega)
\end{split}
\end{equation}
with constants depending only on $\cshape$.
\end{lemma}
\begin{proof}
In this proof we will derive a localized quantification for the residual $\scr{R}$.
In view of definition \eqref{eq:residual} we follow standard procedure and integrate by parts to obtain
 Let
\[
\inner{\scr{R},\v}=\a(\u-\U,\v)\quad\forall\v\in H^2_0(\Omega)
\]
We will derive a localized quantification for the residual $\scr{R}_\mesh$ which will provide a sharp upper-bound estimate for residual.
\begin{equation}
\begin{split}
\inner{\scr{R},\v}=&\sum_{\tau\in\mesh}\left(
\int_\tau (f-\Op\U)\v+\int_{\diff\tau}\Lap\U{\textstyle\dfrac{\v}{\n_\tau}}
-\int_{\diff\tau}{\textstyle\dfrac{\Lap\U}{\n_\tau}\v}\right)
\end{split}
\end{equation}
Expressing all the integrals over cell boundaries as integrals over edges,
\begin{equation}
\begin{split}
\sum_{\tau\in\mesh}\left(\int_{\diff\tau}\Lap\U{\textstyle\dfrac{\v}{\n_\tau}}
-\int_{\diff\tau}{\textstyle\dfrac{\Lap\U}{\n_\tau}\v}\right)
=\sum_{\sigma\in\edges_\mesh}\left(\int_\sigma\jump{\textstyle\dfrac{\Lap\U}{\n_\sigma}}{\sigma}\v
-\int_\sigma\jump{\Lap\U}{\sigma}{\textstyle\dfrac{\v}{\n_\sigma}}\right).
\end{split}
\end{equation}
%
We have
\begin{equation}\label{eq1:lem:er}
\begin{split}
|\!\inner{\scr{R},\v}\!|\leq&\sum_{\tau\in\mesh}\norm{f-\Op\U}{L^2(\tau)}\norm{\v-\Ip\v}{L^2(\tau)}
+\sum_{\sigma\in\edges_\mesh}\left\norm{\jump{\textstyle\dfrac{\Lap\U}{\n_\sigma}}{\sigma}\right}{L^2(\sigma)}\left\norm{\textstyle(\v-\Ip\v)\right}{L^2(\sigma)}\\
&+\sum_{\sigma\in\edges_\mesh}\norm{\jump{\Lap\U}{\sigma}}{L^2(\sigma)}\left\norm{{\textstyle\dfrac{}{\n_\sigma}}(\v-\Ip\v)\right}{L^2(\sigma)}.
\end{split}
\end{equation}
%
We define the interior residual $R=f-\Op\U$ and jump terms $J_1=\textstyle\jump{\dfrac{\Lap\U}{\n_\sigma}}{\sigma}$ and $J_2=\jump{\Lap\U}{\sigma}$ across  each interior edge $\sigma$.
%
Starting with the first three terms in \eqref{eq1:lem:er}, we use the approximation results from Lemma \ref{lem:ie} to estimate interior residual terms
\begin{equation}
\begin{split}
\sum_{\tau\in\mesh}\norm{R}{L^2(\tau)}\norm{\v-\Ip\v}{L^2(\tau)}
&\leq\sum_{\tau\in\mesh}\norm{R}{L^2(\tau)} \cproj h_\tau^2\norm{\v}{L^2(\omega_\tau)},\\
&\leq \cproj \bigg(\sum_{\tau\in\mesh}h_\tau^4\norm{R}{L^2(\tau)}^4\bigg)^{1/2}
\bigg(\sum_{\tau\in\mesh}\norm{\v}{H^2(\omega_\tau)}^2\bigg)^{1/2}.
\end{split}
\end{equation}
%
As for the interior edge jump terms,
\begin{equation}
\begin{split}
\sum_{\sigma\in\edges_\mesh}\left\norm{J_1\right}{L^2(\sigma)}\left\norm{\textstyle(\v-\Ip\v)\right}{L^2(\sigma)}
&\leq\sum_{\sigma\in\edges_\mesh}\left\norm{J_1\right}{L^2(\sigma)}\cproj h_\sigma^{3/2}\norm{\v}{L^2(\omega_\sigma)},\\
&\leq \cproj \bigg(\sum_{\sigma\in\edges_\mesh}h_\sigma^3\norm{J_1}{L^2(\sigma)}^2\bigg)^{1/2}
\bigg(\sum_{\sigma\in\edges_\mesh}\norm{\v}{H^2(\omega_\sigma)}^2\bigg)^{1/2},
\end{split}
\end{equation}
and
\begin{equation}
\begin{split}
\sum_{\sigma\in\edges_\mesh}\left\norm{J_2\right}{L^2(\sigma)}\left\norm{{\textstyle\dfrac{}{\n_\sigma}}(\v-\Ip\v)\right}{L^2(\sigma)}
&\leq\sum_{\sigma\in\edges_\mesh}\left\norm{J_2\right}{L^2(\sigma)}\cproj h_\sigma^{1/2}\norm{\v}{L^2(\omega_\sigma)},\\
&\leq \cproj \bigg(\sum_{\sigma\in\edges_\mesh}h_\sigma\norm{J_2}{L^2(\sigma)}^2\bigg)^{1/2}
\bigg(\sum_{\sigma\in\edges}\norm{\v}{H^2(\omega_\sigma)}^2\bigg)^{1/2}.
\end{split}
\end{equation}
%
From the finite-intersection property \eqref{eq:sr} we have $\sum_{\sigma\in\edges}
\norm{\v}{H^2(\omega_\sigma)}^2\leq \cshape\norm{\v}{H^2(\Omega)}^2$.
and using \eqref{eq:sr} we have $\sum_{\sigma\in\bedges}\norm{\v}{H^2(\tau(\sigma))}^2\leq \cshape\norm{\v}{H^2(\Omega)}^2$.
%
Summing up we arrive at
\begin{equation}
\begin{split}
|\!\inner{\scr{R},\v}\!|&\leq \cproj\cshape \left\{\bigg(\sum_{\tau\in\mesh}h_\tau^4\norm{R}{L^2(\tau)}^2\bigg)^{1/2}
+\bigg(\sum_{\sigma\in\edges_\mesh}h_\sigma^3\left\norm{J_1\right}{L^2(\sigma)}^2\bigg)^{1/2}\right.\\
&\quad\left.+\bigg(\sum_{\sigma\in\edges_\mesh}h_\sigma\norm{J_2}{L^2(\sigma)}^2\bigg)^{1/2}
\right\}\norm{\v}{H^2(\Omega)}
\end{split}
\end{equation}
\end{proof}

\subsection{Global lower bound}
We will define extension operators $E_\sigma:C(\sigma)\to C(\tau)$ for all edges $\sigma$ with $\sigma\subset\diff\tau$.
Let $\hat{\tau}=[0,1]\times[0,1]$ and $\hat{\sigma}=[0,1]\times\{0\}$.
Let $F_\tau:\bb{R}^2\to\bb{R}^2$ be the affine transformation comprising of translation and scaling mapping $\hat{\tau}$ onto $\tau$ and $\hat{\sigma}$ onto $\sigma$.
Define $\hat{E}:C(\hat{\sigma})\to C(\hat{\tau})$ via
\begin{equation}
\hat{E}v(x,y)=v(x)\quad \forall x\in\hat{\sigma},\quad(x,y)\in\hat{\tau},\quad v\in C(\hat{\sigma}).
\end{equation}
To this end, let $\sigma$ be an edge of a cell $\tau\in\mesh$, then define $E_\sigma:C(\sigma)\to C(\tau)$ via
\begin{equation}
E_\sigma v:=[\hat{E}(v\circ F_\tau)]\circ F_\tau^{-1}.
\end{equation}
In other words extending the values of $\v$ from $\sigma$ into $\tau$ along inward $\n_\sigma$.
Let $\psi_\tau$ be any smooth cut-off function with the following properties:
\begin{equation}\label{eq:interiorbubble}
\supp\psi_\tau\subseteq\tau,
\quad\psi_\tau\ge0,
\quad\max_{x\in\tau}\psi_\tau(x)\leq1.
\end{equation}
Furthermore, we will define for $\sigma=\diff\tau_1\cap\diff\tau_2=:D_\sigma$ two $C^1$ cut-off functions $\psi_\sigma$ and $\chi_\sigma$ via
with the following
\begin{equation}
\quad\supp\chi_\sigma\subseteq D_\sigma,
\quad\supp\psi_\sigma\subseteq D_\sigma,
\quad\psi_\tau\ge0,
\quad\max_{x\in\tau}\psi_\tau(x)\leq1,
\end{equation}
such that
\begin{equation}
\textstyle\ \dfrac{\psi_\sigma}{\n_\sigma}\equiv0,
\quad\chi_\sigma\equiv0
\quad \text{and}\quad d_1h_\sigma^{-1}\psi_\sigma\leq\dfrac{\chi_\sigma}{\n_\sigma}\leq d_2h_\sigma^{-1}\psi_\sigma\quad \text{along}\ \sigma.
\end{equation}
Let $\hat{\sigma}$ and $\hat{\tau}$ retain the same meanings as before.
Let $\hat{\tau}_1=\tau$ and let $\hat{\tau}_2=\{(x,y)\in\bb{R}^2:(x,-y)\in\hat{\tau}\}$ and let $\n=(0,-1)$
Now let $F_\sigma:\bb{R}^2\to\bb{R}^2$ be the affine map that maps $\hat{D}$ onto $D_\sigma$ and define
\begin{equation}\label{eq:edgebubbles}
\psi_\sigma=\hat{\psi}\circ F_\sigma^{-1},\quad
\chi_\sigma=\hat{\chi}\circ F_\sigma^{-1}.
\end{equation}
unit normal vector on $\hat{\sigma}$.
Let $g(y)$ be the cubic polynomial satisfying
\begin{equation}
g^{(4)}(y)=0\ \text{for}\ y\in(0,1)
\ \text{such that}\
g'(0)=g(1)=g'(1)=0\ \text{and}\ g(0)=1.
\end{equation}
Put $\phi(y)=g(y)\ds{1}_{[0,1]}+g(-y)\ds{1}_{[-1,0]}$.
Now define $G$ to be the quartic polynomial satisfying
\begin{equation}
G^{(5)}(y)=0\ \text{for}\ y\in(0,1)\ \text{such that}\
G(0)=G(1)=G'(1)=0\
\text{and}\ G'(0)=\phi(0).
\end{equation}
Put $\eta(y)=G(y)\ds{1}_{[0,1]}-G(-y)\ds{1}_{[-1,0]}$.
Finally, let $H=x^2(1-x)^2(1-y^2)^2$ and let $\Phi(x,y)=H(x,y)\ds{1}_{y\ge0}+H(x,-y)\ds{1}_{y\leq 0}$.
Finally set
\begin{equation}
\hat{\psi}(x,y)=\Phi(x,y)\phi(y),\quad\hat{\chi}(x,y)=\Phi(x,y)\eta(y),\quad(x,y)\in\hat{D}
\end{equation}
\begin{lemma}[Localizing estimates]\label{lem:LocalizingEstimates}
Let $\mesh$ be a partition of $\Omega$ satisfying Conditions \eqref{eq:sr}.
Let $\tau$ be a cell in partition $\mesh$.
For a constant $c_m>0$ depending only on polynomial degree $m$,
\begin{equation}\label{eq1:result:lem:LocalizingEstimates}
\norm{q}{L^2(\tau)}^2\leq \cc\int_\tau\psi_\tau q^2\quad\forall q\in\bb{P}_m(\tau),
\end{equation}
Let $\sigma$ be an edge in $\edges_\mesh$ and let $\tau_1$ and $\tau_2$ be cells from $\mesh$ for which $\sigma\subseteq\overline{\tau_1}\cap\overline{\tau_2}$.
We also have
\begin{equation}\label{eq2:result:lem:LocalizingEstimates}
\norm{q}{L^2(\sigma)}^2\leq \cc\int_\sigma\psi_\sigma q^2
\end{equation}
and
\begin{equation}\label{eq3:result:lem:LocalizingEstimates}
\norm{\psi_\sigma^{1/2} E_\sigma q}{L^2(\tau)}\leq \cc h_\sigma^{1/2}\norm{q}{L^2(\sigma)}\quad\forall\tau\in D_\sigma
\end{equation}
holding for every $q\in\bb{P}_m(\sigma)$.
\end{lemma}
%
\begin{proof}
Relations \eqref{eq1:result:lem:LocalizingEstimates} and \eqref{eq2:result:lem:LocalizingEstimates} are proven in the same fashion as in \cite{ainsworth2011posteriori}, \cite{verfurth1994posteriori}.
We focus on \eqref{eq3:result:lem:LocalizingEstimates}.
We prove that $q\mapsto\norm{\psi_\sigma^{1/2} E_\sigma q}{L^2(\tau)}$ is a norm on $\bb{P}_m(\sigma)$.
\begin{equation}
\begin{split}
\norm{E_\sigma q}{L^2(\tau)}
&=|\tau|^{1/2}\norm{\hat{E}(q\circ F_\tau)}{L^2(\hat{\tau})}.
\end{split}
\end{equation}
It clear that $\hat{q}\in\bb{P}_m(\hat{\sigma})$ is identically zero if and only if its extension $\hat{E}\hat{q}$ is identically zero on $\hat{\tau}$.
So $\hat{q}\mapsto\norm{\hat{E}\hat{q}}{L^2(\hat{\tau})}$ is an equivalent norm on $\bb{P}_m(\hat{\sigma})$ we have
\begin{equation}
\norm{\hat{E}(q\circ F_\tau)}{L^2(\hat{\tau})}\leq c\norm{q\circ F_\tau}{L^2{\hat{\sigma}}}=ch_\sigma^{-1/2}\norm{q}{L^2(\sigma)}
\end{equation}
so with $c|\tau|^{1/2}h_\sigma^{1/2}\leq\overline{c}h_\sigma^{1/2}$
\begin{equation}
\norm{\psi_\sigma^{1/2}E_\sigma q}{L^2(\tau)}\leq\overline{c}h_\sigma^{1/2}\norm{q}{L^2(\sigma)}.
\end{equation}
\end{proof}
%
\begin{lemma}[Estimator efficiency]\label{lem:EstimatorEfficiency}
Let $\mesh$ be a partition of $\Omega$ satisfying conditions \eqref{eq:sr}.
The module $\ms{ESTIMATE}$ produces a posteriori error estimate of the discrete solution error such that
\begin{equation}\label{eq:result:lem:EstimatorEfficiency}
\Ceff\,\eta_\mesh^2(\U,\Omega)\leq\enorm{\u-\U}{}^2+\osc_\mesh^2(\Omega).
\end{equation}
with constant $\Ceff$ depending only on $\cshape$.
\end{lemma}
\begin{proof}
The proof is carried by localizing the error contributions coming from the cells residuals $R$ and edge jumps $J_1$ and $J_2$.
For $\tau\in\mesh$ let $\psi_\tau\in H^2_0(\tau)$ be as in \eqref{eq:interiorbubble} and let $\overline{R}$ be a polynomial approximation of $R|_\tau$ by means of an $L^2$-orthogonal projection.
Using the norm-equivalence relation \eqref{eq1:result:lem:LocalizingEstimates} of Lemma
\ref{lem:LocalizingEstimates}
\begin{equation}\label{eq1:lem:EstimatorEfficiency}
\norm{\overline{R}}{L^2(\tau)}^2\leq \cc\int_\tau\overline{R}(\overline{R}\psi_\tau)\leq \cc\norm{\overline{R}}{H^{-2}(\tau)}\norm{\overline{R}\psi_\tau}{H^2(\tau)}.
\end{equation}
From \eqref{eq:inve:lem:ie} and \eqref{eq:interiorbubble}, $\norm{\overline{R}\psi_\tau}{H^2(\tau)}\leq \cc h_\tau^{-2}\norm{\overline{R}}{L^2(\tau)}$ with a constant $\cc>0$ which depends on the polynomial degree $r$ and \eqref{eq1:lem:EstimatorEfficiency} now reads $h_\tau^{2}\norm{\overline{R}}{L^2(\tau)}\leq \cc\norm{\overline{R}}{H^{-2}(\tau)}$.
Then
\begin{equation}\label{eq:interior:lem:EstimatorEfficiency}
\begin{split}
h_\tau^2\norm{R}{L^2(\tau)}&\leq h_\tau^2\norm{\overline{R}}{L^2(\tau)}+h_\tau^2\norm{R-\overline{R}}{L^2(\tau)},\\
&\leq \cc\norm{\overline{R}}{H^{-2}(\tau)}+h_\tau^2\norm{R-\overline{R}}{L^2(\tau)},\\
&\leq \cc\norm{R}{H^{-2}(\tau)}+\cc\norm{R-\overline{R}}{H^{-2}(\tau)}+h_\tau^2\norm{R-\overline{R}}{L^2(\tau)},\\
&\lapprox \cc\norm{\scr{R}}{H^{-2}(\tau)}+h_\tau^2\norm{R-\overline{R}}{L^2(\tau)}.
\end{split}
\end{equation}
Recognizing that $R-\overline{R}=f-\overline{f}$, define $\osc_\mesh(\tau):=h^2_\tau\norm{f-\overline{f}}{L^2(\tau)}$.
We turn our attention to the jump terms across the interior edges.
We begin with the edge residual $J_1$.
Let an edge $\sigma\in\edges_\mesh$ and cells $\tau_1,\tau_2\in\mesh$ be such that
$\sigma\subset\diff\tau_1\cap\diff\tau_2$ and denote $D_\sigma=\overline{\tau_1\cup\tau_2}$.
If $\v\in H^2_0(D_\sigma)$ then
\begin{equation}\label{eq2:lem:EstimatorEfficiency}
\inner{\scr{R},\v}=\int_{D_\sigma}R\v+\int_\sigma J_1\v-\int_\sigma J_2\textstyle\dfrac{\v}{\n_\sigma}.
\end{equation}
Let $\psi_\sigma$ be the bubble function \eqref{eq:edgebubbles} and constantly extend the values of $J_1$ in directions $\pm\n_\sigma$; i.e, into each of $\tau_i$, and set $v_\sigma=\psi_\sigma J_1$.
Then \eqref{eq2:lem:EstimatorEfficiency} reads
\begin{equation}\label{eq3:lem:EstimatorEfficiency}
\cc\norm{J_1}{L^2(\sigma)}^2\leq\int_\sigma\psi_\sigma J_1^2=\inner{\scr{R},\v_\sigma}-\int_{D_\sigma}R\v_\sigma.
\end{equation}
From \eqref{eq:dtrace:lem:ie} and \eqref{eq:edgebubbles} we have the estimates $\norm{\psi_\sigma E_\sigma J_1}{H^2(D_\sigma)}\leq \cinv h_\sigma^{-2}\norm{E_\sigma J_1}{L^2(D_\sigma)}$ and $\norm{\psi_\sigma E_\sigma J_1}{L^2(D_\sigma)}\leq \cc h_\sigma^{1/2}\norm{J_1}{L^2(\sigma)}$
which we apply to \eqref{eq3:lem:EstimatorEfficiency} to obtain
\begin{equation}\label{eq:jump1:lem:EstimatorEfficiency}
\begin{split}
\cc\norm{J_1}{L^2(\sigma)}^2&\leq\left(\cinv h_\sigma^{-2}\norm{\scr{R}}{H^{-2}(D_\sigma)}+\norm{R}{L^2(D_\sigma)}\right)\norm{J_1}{L^2(D_\sigma)},\\
&\leq\cdtrace h_\sigma^{1/2}\left(\cinv h_\sigma^{-2}\norm{\scr{R}}{H^{-2}(D_\sigma)}+\norm{R}{L^2(D_\sigma)}\right)\norm{J_1}{L^2(\sigma)},
\end{split}
\end{equation}
where the the last line follows from \eqref{eq3:result:lem:LocalizingEstimates}.
Now let $\chi_\sigma$ be the function \eqref{eq2:result:lem:LocalizingEstimates}, extend the values of $J_2$ into $D_\sigma$ and set $\w_\sigma=\chi_\sigma J_2$.
We then have
\begin{equation}
\inner{\scr{R},\w_\sigma}=\int_{D_\sigma}R\w_\sigma-h_\sigma^{-1}\int_\sigma\psi_\sigma J_2^2.
\end{equation}
Similarly, we obtain
\begin{equation}\label{eq:jump2:lem:EstimatorEfficiency}
\cc h_\sigma^{-1}\norm{J_2}{L^2(\sigma)}^2\leq\cdtrace h_\sigma^{1/2}\left(\cinv h_\sigma^{-2}\norm{\scr{R}}{H^{-2}(D_\sigma)}+\norm{R}{L^2(D_\sigma)}\right)\norm{J_2}{L^2(\sigma)}.
\end{equation}
We have from \eqref{eq:jump1:lem:EstimatorEfficiency} and
\eqref{eq:jump2:lem:EstimatorEfficiency}
\begin{equation}
\begin{split}
h_\sigma^3\norm{J_1}{L^2(\sigma)}^2+h_\sigma\norm{J_2}{L^2(\sigma)}^2
&\lapprox{\textstyle\frac{\cinv\cdtrace}{\cc}}\norm{\scr{R}}{H^{-2}(D_\sigma)}^2
+{\textstyle\frac{\cdtrace}{\cc}} h_\sigma^4\norm{R}{L^2(D_\sigma)}^2.
\end{split}
\end{equation}
Summing up we have
\begin{equation}
\begin{split}
\eta_\mesh^2(\V,\tau)&=h_\tau^4\norm{R}{L^2(\tau)}^2
+\sum_{\sigma\in\diff\tau}\left(h_\sigma^3\norm{J_1}{L^2(\sigma)}^2+h_\sigma\norm{J_2}{L^2(\sigma)}^2\right),\\
&\leq h_\tau^4\norm{R}{L^2(\tau)}^2+{\textstyle\frac{\cdtrace}{\cc}}\sum_{\sigma\in\edges_\mesh}
\left(\cinv\norm{\scr{R}}{H^{-2}(D_\sigma)}^2+h_\sigma^4\norm{R}{L^2(D_\sigma)}^2\right),\\
&\leq (1+\cshape)h_\tau^4\norm{R}{L^2(\omega_\tau)}^2+{\textstyle\frac{\cinv\cdtrace}{\cc}}\sum_{\sigma\in\diff\tau}\norm{\scr{R}}{H^{-2}(D_\sigma)}^2,\\
\end{split}
\end{equation}
we arrive at
\begin{equation}
\begin{split}
\eta_\mesh^2(\V,\tau)&\lapprox(1+\cshape)\left(\cc\norm{\scr{R}}{H^{-2}(\tau)}^2+\osc^2_\mesh(\omega_\tau)\right)
+{\textstyle\frac{\cinv\cdtrace}{\cc}}\sum_{\sigma\in\diff\tau}\norm{\scr{R}}{H^{-2}(D_\sigma)}^2.
\end{split}
\end{equation}
Note that
\begin{equation}
\begin{split}
\sum_{\sigma\in\edges_\mesh}\norm{\scr{R}}{H^{-2}(D_\sigma)}^2
&\leq\Ccont\sum_{\sigma\in\edges_\mesh}\norm{\u-\U}{H^2(D_\sigma)}^2
\leq\cshape\Ccont\norm{\u-\U}{H^2(\Omega)}^2
\end{split}
\end{equation}
\end{proof}

\subsection{Discrete upper bound}
Here we show that the estimator is capable of local quantification of the difference between two consequetive discrete spline solutions.
\begin{lemma}[Estimator discrete reliability]\label{lem:dre}
Let $\mesh$ be a partition of $\Omega$ satisfying conditions \eqref{eq:sr} and let $\mesh_\ast=\ms{REFINE}\,[\mesh,R]$ for some refined set $R\subseteq\mesh$.
If $\U$ and $\U_\ast$ are the respective solutions to \eqref{eq:dcp} on $\mesh$ and $\mesh_\ast$, then for a constants $ C_\rm{dRel,1}, C_\rm{dRel,2}>0$, depending only on $\cshape$,
\begin{equation}\label{eq:result:lem:dre}
\begin{split}
\enorm{\U_\ast-\U}{}^2&\leq C_\rm{dRel,1}\eta^2_\mesh(\U,\omega_{R_{\mesh\to\mesh_\ast}})
\end{split}
\end{equation}
where $\omega_{R_{\mesh\to\mesh_\ast}}$ is understood as the union of support extensions of refined cells from $\mesh$ to obtain $\mesh_\ast$.
\end{lemma}

\begin{proof}
Let $\e_\ast=\U_\ast-\U$. First note that if $\V\in\Xp$ then in view of the nesting $\Xp\subset\Xpp$, %
\begin{equation}
\a(\U_\ast-\U,\e_\ast)=\a(\U_\ast-\U,\e_\ast-\V).
\end{equation}
%
To localize, we form disconnected subdomains $\Omega_i\subseteq\Omega$, $i\in J$, each formed from the interiors of connected components of $\Omega_\ast=\cup_{\tau\in R_{\mesh\to\mesh_\ast}}\overline{\tau}$.
Then to each subdomain $\Omega_i$ we form a partition $\mesh_i=\{\face\in\mesh:\face\subset\Omega_i\}$, interior edges $\edges_i=\{\sigma\in\edges_\mesh:\sigma\subset\diff\tau,\ \tau\in\mesh_i\}$, and a corresponding finite-element space $\X_i$. Let $I_i:H^2(\Omega_i)\to\X_i$.
%
Let $\V\in\Xp$ be an approximation of $\e_\ast$ be given by
\begin{equation}
\V=\e_\ast\ds{1}_{\Omega\backslash\Omega_\ast}+\sum_{i\in J}(I_i\e_\ast)\cdot\ds{1}_{\Omega_i}.
\end{equation}
%
Then $e_\ast-\V\equiv0$ on $\Omega\backslash\Omega_\ast$ and performing  integration by parts will yield
\begin{equation}
\begin{split}
\a(\U_\ast-\U,\e_\ast-\V)=&\sum_{i\in J}
\bigg[\sum_{\tau\in\mesh_i}\inner{R,\e_\ast-\Ip\e_\ast}_\tau
+\sum_{\sigma\in\edges_i}\left\{\inner{J_1,\e_\ast-I_i\e_\ast}_\sigma+\inner{J_2,\e_\ast-I_i\e_\ast}_\sigma\right\}\bigg],
\end{split}
\end{equation}
%
Following the same procedure carried in Lemma \ref{lem:er} we have
\begin{equation}
\begin{split}
\sum_{\tau\in\mesh_i}\inner{R,\e_\ast-I_i\e_\ast}_\tau
&+\sum_{\sigma\in\edges_i}\left\{\inner{J_1,\e_\ast-I_i\e_\ast}_\sigma
+\inner{J_2,\e_\ast-I_i\e_\ast}_\sigma\right\}\\
&\leq\cproj\bigg(\sum_{\tau\in\mesh_i}\eta^2_\mesh(\U,\tau)\bigg)^{1/2}
\bigg(\sum_{\tau\in\mesh_i}\norm{\e_\ast}{H^2(\omega_\tau)}^2\bigg)^{1/2},
\end{split}
\end{equation}
Set $\omega_{R_{\mesh\to\mesh_\ast}}=\cup\{\omega_\tau:\tau\in R_{\mesh\to\mesh_\ast}\}$.
%
We therefore have
\begin{equation}
\begin{split}
\enorm{\U_\ast-\U}{}^2&\leq\Cdrel\eta_\mesh(\U,\omega_{R_{\mesh\to\mesh_\ast}})\norm{\e_\ast}{H^2(\Omega)}
\end{split}
\end{equation}
\end{proof}

\section{Convergence} 		
In section we show that the derived computable estimator \eqref{eq:estimator} when used to direct refinement will result in decreased error.
This will hinge on the estimator Lipschitz property of Lemma \ref{lem:elp}.
To show that procedure \eqref{eq:afem} exhibits convergence we must be able to relate the errors of consecutive discrete solutions.
The symmetry of the bilinear form, consistency of the formulation and  finite-element spline space nesting will readily provide that via Galerkin Pythagoras in Lemma \ref{lem:gp}.
%
%
\subsection{Error reduction} 
\begin{lemma}[Estimator Lipschitz property]\label{lem:elp}
Let $\mesh$ be a partition of $\Omega$ satisfying conditions \eqref{eq:sr}.
There exists a constant $\Clip>0$, depending only $\cshape$, such that for any cell $\face\in\mesh$ we have
\begin{equation}\label{eq:result:lem:elp}
|\eta_\mesh (\V,\face)-\eta_\mesh (\W,\face)|
\leq\Clip\semi{\V-\W}{H^2(\omega_\face)},
\end{equation}
holding for every pair of finite-element splines $\V$ and $\W$ in $\X_\mesh$.
\end{lemma}
\begin{proof}
Let $\V$ and $\W$ be finite-element splines in $\X_\mesh$ and let $\face$ be a cell in partition $\mesh$.
\begin{equation}\label{eq1:lem:elp}
\begin{split}
\eta_\mesh(\V,\face)-\eta_\mesh(\W,\face)&=
h_\face^2\left(\norm{\eff-\Op\V}{L^2(\face)}
-\norm{\eff-\Op\W}{L^2(\face)}\right)\\
&+\sum_{\sigma\subset\diff\tau}h_\sigma^{1/2}\left(
\norm{\jump{\Lap\V}{\sigma}}{L^2(\sigma)}
-\norm{\jump{\Lap\W}{\sigma}}{L^2(\sigma)}
\right)\\
&+\sum_{\sigma\subset\diff\tau}h_\sigma^{3/2}\left(
\norm{\jump{\diffe\Lap\V}{\sigma}}{L^2(\edge)}
-\norm{\jump{\diffe\Lap\W}{\sigma}}{L^2(\edge)}
\right).
\end{split}
\end{equation}
Treating the interior term,
\begin{equation}
\begin{split}
\norm{\eff-\Op\V}{L^2(\face)}-\norm{\eff-\Op\W}{L^2(\face)}
&\leq\semi{\V-\W}{H^4(\face)}
\leq\cinv h^{-2}_\face\semi{\V-\W}{H^2(\face)}.
\end{split}
\end{equation}
Treating the edge terms we have
\begin{equation}
\norm{\jump{\Lap\V}{\sigma}}{L^2(\sigma)}
-\norm{\jump{\Lap\W}{\sigma}}{L^2(\sigma)}
\leq\norm{\jump{\Lap\V-\Lap\W}{\sigma}}{L^2(\sigma)}.
\end{equation}
Let $\face'$ from $\mesh$ be a cell that shares the edge $\sigma$, i.e $\tau'$ is an adjacent cell to $\face$.
For any finite-element spline $\V\in\Xp$ we have
\begin{equation}
\norm{\jump{\V}{\edge}}{\edge}
\leq\cdtrace\left(
h_\sigma^{-1/2}\norm{\V}{\face}+h_\sigma^{-1/2}\norm{\V}{\face'}
\right)
\leq\cdtrace h_\sigma^{-1/2}\norm{\V}{\omega_\face}.
\end{equation}
Replacing $\V$ with $\Lap\V-\Lap\W$ gives
\begin{equation}
h_\sigma^{1/2}\norm{\jump{\Lap\V-\Lap\W}{\sigma}}{\sigma}
\leq\cinv\cdtrace\semi{\V-\W}{H^2(\omega_\face)}.
\end{equation}
Similarly, we have
\begin{equation}
h_\sigma^{3/2}\left(
\norm{\jump{\diffe\Lap\V}{\sigma}}{L^2(\edge)}
-\norm{\jump{\diffe\Lap\W}{\sigma}}{L^2(\edge)}
\right)
\leq\cinv\cdtrace\semi{\V-\W}{H^2(\omega_\face)}.
\end{equation}
It then follows from \eqref{eq1:lem:elp}
\begin{equation}
\begin{split}
|\eta_\mesh(\V,\tau)-\eta_\mesh(\W,\tau)|&\leq
\cinv(\semi{\V-\W}{H^2(\face)}+2\cdtrace\semi{\V-\W}{H^2(\omega_\face)}),\\
&\leq\cinv(1+2\cdtrace)\semi{\V-\W}{H^2(\omega_\face)}.
\end{split}
\end{equation}
\end{proof}

\begin{lemma}[Estimator error reduction]\label{lem:err}
Let $\mesh$ be a partition of $\Omega$ satisfying conditions \eqref{eq:sr},
let $\marked\subseteq\mesh$
and let $\mesh_\ast=\mathbf{REFINE}\,[\mesh,\marked]$.
There exists constants $\lambda\in(0,1)$ and $\Cest>0$, depending only on $\cshape$,
such that for any $\delta>0$ it holds that for any pair of finite-element splines $\V\in\X_\mesh$ and $\V_\ast\in\X_{\mesh_\ast}$
we have
\begin{equation}\label{eq:result:lem:err}
\eta_{\mesh_\ast}^2(\V_\ast,\Omega)
\leq(1+\delta)\left\{\eta_\mesh^2(\V,\Omega)-{\textstyle\frac{1}{2}}\eta_\mesh^2(\V,\marked)\right\}+\cshape(1+{\textstyle\frac{1}{\delta}})\enorm{\V-\V_\ast}{}^2.
\end{equation}
\end{lemma}
\begin{proof}
Let $\marked\subseteq\mesh$ be a set of marked elements from partition $\mesh$
and let $\mesh_\ast=\mathbf{REFINE}\,[\mesh,\marked]$.
For notational simplicity we denote $\bb{X}_{\mesh_\ast}$ and $\eta_{\mesh_\ast}$ by $\X_\ast$ and $\eta_\ast$, respectively.
Let $\V$ and $\V_\ast$ be the respective finite-element splines from $\Xp$ and $\X_\ast$.
Let $\face$ be a cell from partition $\mesh_\ast$.
In view of the Lipschitz property of the estimator (Lemma \ref{lem:elp})
and the nesting $\Xp\subseteq\X_\ast$,
\begin{equation}
\eta_{\ast}^2(\V_\ast,\face)
\lapprox\eta_{\ast}^2(\V,\face)+\semi{\V-\V_\ast}{H^2(\omega_\face)}^2
+2\eta_{\ast}(\V_\ast,\face)\semi{\V-\V_\ast}{H^2(\omega_\face)}.
\end{equation}
Given any $\delta>0$, an application of Young's inequality on the last term gives
\begin{equation}
2\eta_{\ast}(\V_\ast,\face)\semi{\V-\V_\ast}{H^2(\omega_\face)}
\leq\delta\eta_{\ast}^2(\V_\ast,\face)+{\textstyle\frac{1}{\delta}}\semi{\V-\V_\ast}{H^2(\omega_\face)}^2.
\end{equation}
We now have
\begin{equation}
\eta_{\ast}^2(\V_\ast,\face)
\lapprox(1+\delta)\eta_{\ast}^2(\V,\face)
+(1+{\textstyle\frac{1}{\delta}})\semi{\V-\V_\ast}{H^2(\omega_\face)}^2.
\end{equation}
Recalling that the partition cell are disjoint with uniformly bounded support extensions, we may sum over all the cells $\face\in\mesh_\ast$ to obtain
\begin{equation}
\eta_\ast^2(\V_\ast,\mesh_\ast)\leq(1+\delta)\eta_\ast^2(\V,\mesh_\ast)+\cshape(1+{\textstyle\frac{1}{\delta}})\enorm{\V-\V_\ast}{}^2.
\end{equation}
It remains to estimate $\eta_\ast^2(\V,\mesh_\ast)$. Let $|\marked|$ be the sum areas of all cells in $\marked$.
For every marked element $\face\in\marked$ define $\mesh_{\ast,\marked}=\{\rm{child}(\face):\face\in\marked\}$.
Let $b>0$ denote the number of bisections required to obtain the conforming partition $\mesh_\ast$ from $\mesh$.
Let $\face_\ast$ be a child of a cell $\face\in\marked$.
Then $h_{\face_\ast}\leq2^{-1}h_\face$.
Noting that $\V\in\Xp$ we have no jumps within $\face$
\begin{equation}
\eta_\ast^2(\V,\tau_\ast)=h_{\tau_\ast}^4\norm{f-\Op\V}{\tau_\ast}^2\leq(2^{-1}h_\tau)^4\norm{f-\Op\V}{\tau_\ast}^2,
\end{equation}
summing over all children
\begin{equation}
\sum_{\tau_\ast\in\rm{children}(\tau)}\eta_\ast^2(\V,\tau_\ast)
\leq2^{-1}\eta_\mesh^2(\V,\tau),
\end{equation}
and we obtain by disjointness of partitions an estimate on the error reduction
\begin{equation}
\sum_{\face_\ast\in\mesh_{\ast,M}}\eta_{\ast}^2(\V,\face_\ast)\leq2^{-1}\est_\mesh^2(\V,\marked).
\end{equation}
For the remaining cells $T\in\mesh\backslash\marked$, the estimator monotonicity implies $\eta_{\mesh_\ast}(\V,T)\leq\eta_\mesh(\V,T)$.
Decompose the partition $\mesh$ as union of marked cells in $\marked$ and their complement $\mesh\backslash\marked$ to conclude the total error reduction obtained by $\ms{REFINE}$ and the choice of Dorfler parameter $\theta$
\begin{equation}
\eta_{\mesh_\ast}^2(\V,\Omega)\leq\estP^2(\V,\Omega\backslash\marked)+2^{-1}\estP^2(\V,\marked)
=\estP^2(\V,\Omega)-{\textstyle\frac{1}{2}}\eta^2_\mesh(\V,\marked).
\end{equation}
\end{proof}

\begin{lemma}[Galerkin Pythaguras]\label{lem:gp}
Let $\mesh$ and $\mesh_\ast$ be partitions of $\Omega$ satisfying conditions \eqref{eq:sr} with $\mesh_\ast\ge\mesh$
and let $\U\in\Xp$ and $\U_\ast\in\X_{\mesh_\ast}$ be the spline solutions to \eqref{eq:dcp}.
Then
\begin{equation}\label{eq:result:lem:gp}
\enorm{\u-\U_\ast}{}^2=\enorm{\u-\U_\ast}{}^2-\enorm{\U_\ast-\U}{}^2.
\end{equation}
\end{lemma}
\begin{proof}
At first we express
\begin{equation}\label{eq1:lem:gp}
\begin{split}
\a(\u-\U_\ast,\u-\U_\ast)&=
\a(\u-\U,\u-\U)-\a(\upp-\up,\upp-\up)\\
&+\a(\up-\upp,\u-\upp)+\a(\u-\U_\ast,\U-\U_\ast).
\end{split}
\end{equation}
Recognizing that $\a(\u-\U_\ast,\U-\U_\ast)=\a(\U-\U_\ast,\u-\U)=0$, we arrive at
\begin{equation}
\a(\u-\U_\ast,\u-\U_\ast)=\a(\u-\U,\u-\U)-\a(\upp-\up,\upp-\up).
\end{equation}
\end{proof}

\begin{theorem}[Convergence of conforming AFEM]\label{thm:ConvConf}
For a contractive factor $\alpha\in(0,1)$ and a constant $\Cest>0$, given any successive mesh partitions $\mesh$ and $\mesh_\ast$ satisfying conditions \eqref{eq:sr}, $f\in L^2(\Omega)$ and Dolfer parameter $\theta\in(0,1]$, the adaptive procedure $\mathbf{AFEM}\,[\mesh,f,\theta]$ with produce two successive solutions $\U\in\Xp$ and $\U_\ast\in\Xpp$ to problem \eqref{eq:dcp} for which
\begin{equation}\label{eq:result:thm:ConvConf}
\enorm{\u-\U_\ast}{}^2+\Cest\eta_{\mesh_\ast}^2(\U_\ast,\Omega)\leq\alpha\big(\enorm{\u-\U}{}^2+\Cest\eta_{\mesh}^2(\U,\Omega)\big).
\end{equation}
\end{theorem}
\begin{proof}
Adopt the following abbreviations:
\begin{eqnarray}
e_\mesh=\enorm{\u-\U}{},&E_\ast=\enorm{\U_\ast-\U}{},\\
\eta_\mesh=\eta_{\mesh}(\U,\Omega),&\eta_\mesh(\marked)=\eta_{\mesh}(\U,\marked).
\end{eqnarray}
Define constants $\qest(\theta,\delta):=(1+\delta)(1-{\frac{\theta^2}{2}})$ and $\Cest^{-1}:=\cshape(1+{\frac{1}{\delta}})$ so that in view of Dorlfer $-\eta_\mesh^2(\marked)\leq-\theta^2\eta_\mesh^2$,
\begin{equation}
(1+\delta)\left\{\eta_\mesh^2(\Omega)-{\textstyle\frac{1}{2}}\eta_\mesh^2(\marked)\right\}\leq q_\rm{est}\eta_\mesh^2,
\end{equation}
and \eqref{eq:result:lem:err} reads $\eta_{\mesh_\ast}^2\leq\qest\eta_\mesh^2+\Cest^{-1} E_\ast^2$.
Together with Galerkin orthogonality,
\begin{equation}
\begin{split}
e_{\mesh_\ast}^2+\Cest\eta_{\mesh_\ast}^2&\leq e_\mesh^2-E^2_\ast+\Cest\left(\qest\eta_\mesh^2+\Cest^{-1} E_\ast^2\right)=e_{\mesh}^2+\qest\Cest\eta_\mesh^2
\end{split}
\end{equation}
Let $\alpha$ be a positive parameter and express $e_{\mesh}^2=\alpha e_{\mesh}^2+(1-\alpha)e_{\mesh}^2$. Invoking on the reliability estimate $e_{\mesh}^2\leq \Crel\eta_\mesh^2$ on one of the decomposed terms gives
\begin{equation}
e_{\mesh_\ast}^2+\Cest\eta_{\mesh_\ast}^2\leq\alpha e_{\mesh}^2+\left[(1-\alpha)\Crel+\qest\Cest\right]\eta_\mesh^2.
\end{equation}
Choose $\delta>0$ with $\delta<\frac{\theta^2}{2-\theta^2}$ so that $q_\rm{est}\in(0,1)$ and we may choose a contractive $\alpha<1$ for which $(1-\alpha)\Crel+\qest\Cest=\textstyle\alpha \Cest$.
Indeed,
\begin{equation}
\alpha=\frac{\qest\Cest+\Crel}{\Cest+\Crel}<1.
\end{equation}
\end{proof}
\begin{remark}\label{rem:ConvergenceAlpha}
Observe that if $C(\delta):=\Crel/\Cest(\delta),$
\begin{equation}
\alpha=\frac{(1+\delta)(1-\frac{\theta^2}{2})+C}{1+C}
=1-\frac{1+\delta}{1+C}\left(\frac{\theta^2}{2}-\frac{\delta}{1+\delta}\right),
\end{equation}
and $(1+\delta)/(1+C)\to0$ as $\delta\to0$. Then it is clear that contractive factor $\alpha$ deteriorates as $\theta\to0$. We have $\alpha\approx1-c\theta^2$ for some constant $c<1$ depending on $\delta$.
\end{remark}
\section{Quasi-optimlaity of AFEM}
The selection of marked cells within every loop of the AFEM dictated by procedure \textbf{MARK} is determined by the error indicators \eqref{eq:indicator}. Hence the decay rate of the adaptive method in terms of the DOFs is heavily dependant on the estimator \eqref{eq:estimator} which, in view of the Global Upper Bound \eqref{lem:er}, may exhibit slower decay than the enrgy norm whenever overresatimation occurs.
In view of the Global Lower Bound \eqref{lem:EstimatorEfficiency}, the quality of the estimator strongly depends on the resolution of the right-hand-side source function $f$ on the mesh resulting from the averaging process the finite element approximation yields manifesting in the oscillation term $\osc_\mesh(f,\Omega)$. 
The fact that the estimator is equivalent to the error in the energy norm up-to the oscillation term
\begin{equation}
\Ceff\eta^2(\U,\Omega)-\osc_\mesh^2(f,\Omega)\leq\enorm{\u-\U}{}^2\leq\Crel\eta_\mesh^2(\U,\Omega)
\end{equation}
motivates measuring the decay rate of the \emph{total-error}
\begin{equation}\label{eq:TotalError}
\rho_\mesh(\v,\V,g):=\left(\enorm{\v-\V}{}^2+\osc^2_\mesh(g,\Omega)\right)^{1/2},
\end{equation}
which in the asymptotic regime, due to estimator dominance over oscillation, can be made equivalent to the quasi-error
\begin{equation}
\rho^2_\mesh(\u,\U,f)\approx\eta^2_\mesh(\U,\mesh)\approx\enorm{\u-\U}{}^2+\Cest\eta_\mesh^2(\U,\mesh)
\end{equation}
In what follows we show that Cea's lemma holds for the total-error norm, that is, that the finite element solution $\U$ is an optimal choice from $\Xp$ in total-error norm. 
\begin{lemma}[Optimality of total error]\label{lem:OptimalityTotalError}
Let $\u$ be the solution of \eqref{eq:cwp} and for all $\mesh\in\scr{P}$ let $\U\in\Xp$ be the discrete solution to \eqref{eq:dcp}. Then,
\begin{equation}
\rho^2_\mesh(\u,\U,f)\leq\inf_{\V\in\Xp}\rho^2_\mesh(\u,\V,f).
\end{equation}
\end{lemma}
\begin{proof}
In view of Galerkin orthogonality and the symmetry of the bilinear form $\a(\u-\U,\U-\V)=\a(\U-\V,\u-\U)=0$ we have
\begin{equation}
\begin{split}
\a(\u-\V,\u-\V_\eps)&=\a(\u-\U,\u-\U)+\a(\U-\V,\U-\V)
\end{split}
\end{equation}
and we have $\enorm{\u-\V}{}^2=\enorm{\u-\U}{}^2+\enorm{\U-\V}{}^2$.
Therefore,
\begin{equation}
\begin{split}
\rho_\mesh^2(\u,\U,f)&\leq\enorm{\u-\U}{}^2+\osc^2_\mesh(f,\Omega)+\enorm{\U-\V}{}^2\\
&=\enorm{\u-\V}{}^2+\osc^2_\mesh(f,\Omega)\\
&=\rho_\mesh^2(\u,\V,f)
\end{split}
\end{equation}
\end{proof}
\subsection{AFEM approximation class}
In order to assess the perfomance of AFEM \eqref{eq:afem}, the rate of decay of error in terms of DOFs, we consider the following nonlinear approximation classes that govern the adaptive finite element problem:
\begin{equation}
\cal{A}_s=\left\{\v\in H^2_0(\Omega):\sup_{N>0}N^s\inf_{\mesh\in\scr{P}_N}\inf_{\V\in\Xp}\enorm{\v-\V}{}<\infty\right\}
\end{equation}
\begin{equation}
\cal{O}_s=\left\{g\in L^2(\Omega):\sup_{N>0}N^s\inf_{\mesh\in\scr{P}_N}\norm{h_\mesh^2(g-\Pi g)}{L^2(\Omega)}<\infty\right\}
\end{equation}
If $(\u,f)\in\cal{A}_s\times\cal{O}_s$ then nonlinear-approximation theory dictates there exists an admissible partition $\mesh\in\scr{P}_N$ for which $\u$ can be approximated in $\Xp$ with an error proprtional to $N^{-s}$. We hope that the proposed AFEM \eqref{eq:afem} will generate a sequence of partitions $\mesh_\ell$ for which $\rho_{\mesh_\ell}^2(\u,\U_\ell,f)$ decays with order $(\#\mesh_\ell)^{-s}$. We define the approximation class described by the total-error norm \eqref{eq:TotalError}
\begin{equation}
\bb{A}_s=\left\{v\in H^2_0(\Omega):\semi{v}{\bb{A}_s}:=\sup_{N>0}N^s\inf_{\mesh\in\scr{P}_N}E_\mesh(\v)<\infty\right\}
\end{equation}
where
\begin{equation}
E_\mesh(\v)=\inf_{\V\in\Xp}\rho_\mesh(\v,\V,\cal{L}\v),\quad\v\in H^2_0(\Omega)
\end{equation}
\begin{remark}
We will restrict values $s\in(0,r/2]$. For values $s>r/2$ resulting approximation spaces will consist of polynomials only. Valid values for rate $s$ will be refined and made more precise when we characterize the aforementioned approximation classes in terms of smoothness.
\end{remark}
\begin{lemma}[Equivalence of classes]\label{lem:EquivClasses}
Let $\u$ be the weak solution to \eqref{eq:cwp}. If $\u\in\cal{A}_s$ and $f\in\cal{O}_s$ then $\u\in\bb{A}_s$.
\end{lemma}
\begin{proof}
By assumption we have two admissible partitions $\mesh_1,\mesh_2\in\scr{P}_N$ and a finite-element spline $\V\in\X_{\mesh_1}$ such that $\enorm{\u-\V}{}\lapprox N^{-s}$ and $\osc_{\mesh_2}(f)\lapprox N^{-s}$. Invoking Mesh Overlay \eqref{eq:MeshOverelay} we obtain an admissible partition $\mesh:=\mesh_1\oplus\mesh_2$ for which $\#\mesh\lapprox 2N$ and because of spline space nesting we have
\[
\enorm{\u-\V}{}+\osc_{\mesh}(f)\lapprox N^{-s}.
\]
\end{proof}
\subsection{Quasi-optimality}
The contraction achieved in the convergence proof is ensured by the Dorfler marking strategy. However the relationship between the Dorfler strategy and error reduction in the total-error norm goes deeper than asserted in Theorem \ref{thm:ConvConf}. In the following lemma we show that if $R_{\mesh\to\mesh_\ast}$ is a set of refined elements resultiing in a reduction of error in contractive sense, then necessarily the Dorfler property holds for  the set $\omega_{R_{\mesh\to\mesh_\ast}}$. The fact will be instrumental in proving that the cardinality of marked cells will keep the partition cardinality at each refinement step proprtional to the optimal quantity dictated by nononlinear approximation.
\begin{lemma}[Optimal Marking]\label{lem:OptimalMarking}
Let $\U=\ms{SOLVE}\,[\mesh,f]$, let $\mesh_\ast$ be any refinement of $\mesh$ and let $\U_\ast=\ms{SOLVE}\,[\mesh_\ast,\eff]$. If for some positive $\mu<1$
\begin{equation}\label{eq:lem:OptimalMarking:assumption}
\enorm{u-\U_\ast}{}^2+\rm{osc}_\ast^2(\eff,\mesh_\ast)\leq\mu\big(\enorm{u-\U}{}^2+\osc_\mesh^2(\eff,\mesh)\big),
\end{equation}
and $R_{P\to P_\ast}$ denotes collection of all elements in $P$ requiring refinement to obtain $P_\ast$ from $P$, then for $\theta\in(0,\theta_\ast)$ we have 
\begin{equation}
\eta_\mesh (\U,\omega_{R_{P\to P_\ast}})\ge\theta\eta_\mesh (\U,P)
\end{equation}
\end{lemma}
\begin{proof}
Let $\theta<\theta_\ast$, the parameter $\theta_\ast$ to be specified later, such that the linear contraction of the total error holds for $\mu:=1-\frac{\theta^2}{\theta_\ast^2}>0$. 
The Efficiency Estimate \eqref{eq:result:lem:EstimatorEfficiency} together with the assumption \eqref{eq:lem:OptimalMarking:assumption}
\begin{equation}
\begin{split}
(1-\mu)\Ceff\eta_\mesh^2(\U,\mesh)&\leq(1-\mu)\rho_\mesh^2(\u,\U,f)\\
&=\rho_\mesh^2(\u,\U,f)-\rho_\ast^2(\u_\ast,\U_\ast,f)\\
&=\enorm{\u-\U}{}^2-\enorm{\u-\U_\ast}{}^2+\osc_P^2(f,\Omega)-\osc_{\mesh_\ast}^2(f,\Omega)
\end{split}
\end{equation}
In view of Galerkin pythagorus gives $\enorm{\u-\U}{}^2-\enorm{\u-\U_\ast}{}^2=\enorm{\U-\U_\ast}{}^2$.
$R_{\mesh\to\mesh_\ast}\subset\mesh$ so $\osc_P^2(f,\Omega)-\osc_{\mesh_\ast}^2(f,\Omega)\leq\osc_\mesh^2(f,\omega_{R_{\mesh\to\mesh_\ast}})$.
Estimator asymptotic dominance over oscillation $\osc_\mesh^2(\U,\tau)\leq\eta_\mesh^2(\U,\tau)$ and Discrete Upper Bound \eqref{eq:result:lem:dre}
\begin{equation}
(1-\mu)\Ceff\eta_\mesh^2(\U,\mesh)\leq(1+\Cdrel)\eta_\mesh^2(\U,\omega_{R_{\mesh\to\mesh_\ast}})
\end{equation}
By definition $\theta^2=(1-\mu)\theta_\ast^2<\theta_\ast^2$ we arrive at $\theta^2\eta_\mesh^2(\U,\mesh)\leq\eta_\mesh^2(\U,\omega_{R_{\mesh\to\mesh_\ast}})$ for $\theta^2<\frac{\Ceff}{1+\Cdrel}=:\theta_\ast^2$.
\end{proof}
\begin{lemma}[Cardinality of Marked Cells]\label{lem:MarkedComplexity}
Let $\{(\mesh_\ell,\bb{X}_\ell,\U_\ell)\}_{\ell\ge0}$ be sequence generated by $\ms{AFEM}\,(\mesh_0,\eff;\eps,\theta)$ for admissible $P_0$ and the pair $\u\in\bb{A}^s$ for some $s>0$ then 
\begin{equation}\label{eq:lem:MarkedComplexity}
\#\scr{M}_\ell\lapprox\left(1-\frac{\theta^2}{\theta_\ast^2}\right)^{-\frac{1}{2s}}\semi{\u}{\bb{A}_s}^{-\frac{1}{s}}\bigg\{\enorm{\u-\U_\ell}{}^2+\rm{osc}_\ell^2(\eff,\mesh_\ell)\bigg\}^{-\frac{1}{2s}}
\end{equation}
\end{lemma}
\begin{proof}
Assume that the marking parameter satisfies the hypothesis of Theorem \ref{lem:OptimalMarking} and suppose that $u\in\bb{A}_s$ for some $s>0$.
Set $\mu=1-\frac{\theta^2}{\theta_\ast^2}$ and let $\eps:=\mu\rho_\ell(\u,\U_\ell,f)^2$. Then by definition of $\bb{A}_s$ there exists an admissible partition $P_\eps$ and a spline $\V_\eps\in\X_\eps$ for which
\begin{equation}
\rho^2_\eps(\u,\V_\eps,f)\leq\eps^2\quad\text{with}\quad\#\mesh_\eps-\#\mesh_0\lapprox|u|_{\bb{A}_s}^{1/s}\eps^{-1/s}.
\end{equation}
Let $\mesh_\ast:=\mesh_\eps\oplus\mesh_\ell$ be the overlay partition of $\mesh_\eps$ and $\mesh_\ell$, $\ell\ge0$, and let $\U_\ast\in\X_\ast$ be the corresponding spline solution.
In view of Optimality of Total Error in Lemma \ref{lem:OptimalityTotalError} and the fact $\mesh_\ast\ge\mesh_\eps$ makes $\X_\ast\supseteq\X_\eps$ and
\begin{equation}
\rho^2_\ast(\u,\U_\ast,f)\leq\rho^2_\eps(\u,\V_\eps,f)\leq\eps^2=\mu\rho^2_\ell(\u,\U_\ell,f)
\end{equation}
From Optimal Marking of Lemma \ref{lem:OptimalMarking} we have $R_{\mesh_\ell\to\mesh_\ast}\subset\mesh_\ell$ satisfying Dorfler property for $\theta<\theta_\ast$.
\begin{equation}
\#\scr{M}_\ell\leq\#R_{\mesh_\ell\to\mesh_\ast}\leq\#\mesh_\ast-\#\mesh_\ell
\end{equation}
In view of overlay property $\#\mesh_\ast\leq\mesh_\eps+\#\mesh_\ell-\#\mesh_0$ in \eqref{eq:MeshOverelay} and definition of $\eps$ we arrive at
\begin{equation}
\#\scr{M}_\ell\leq\#\mesh_\eps-\#\mesh_0\lapprox\mu^{-1/2s}|u|_{\bb{A}_s}^{1/s}\rho_\ell(\u,\U_\ell,f)^{-1/s}
\end{equation}
\end{proof}
\begin{theorem}[Quasi-optimality]
If $\u\in\bb{A}^s$ and $P_0$ is admissible, then the call $\mathbf{AFEM}\,[\mesh_0,\eff,\eps,\theta]$ generates a sequence $\{(\mesh_\ell,\bb{X}_\ell,\U_\ell)\}_{\ell\ge0}$ of strictly admissible partitions $\mesh_\ell$, conforming finite-element spline spaces $\bb{X}_\ell$ and discrete solutions $\U_\ell$ satisfying
\begin{equation}
\rho_\ell(\u,\U_\ell,\eff)
\lapprox\Phi(,\theta)\semi{(\u,\eff)}{\bb{A}_s}^{}(\#\mesh-\#\mesh_0)^{-s}
\end{equation}
with $\Phi(,\theta)=(1-{\theta^2}/{\theta_\ast^2})^{-{1}/{2}}$
\end{theorem}
\begin{proof}
Let $\theta<\theta_\ast$ be given and assume that $u\in\bb{A}^s(\rho)$. 
We will show that the adaptive procedure $\ms{AFEM}$ will produce a sequence $\{(\mesh_\ell,\X_\ell,\U_\ell)\}_{\ell\ge0}$ such that $\rho_\ell\lapprox(\#\mesh_\ell-\#\mesh_0)^{-s}$.
Let $A(\theta,s):=(1-{\theta^2}/{\theta_\ast^2})^{-{1}/{2s}}|u|_{\bb{A}^s}^{-{1}/{s}}$
Cardinality of Marked Cells \eqref{eq:lem:MarkedComplexity} and \eqref{eq:MarkingComplexity} yields
\[
\#\mesh_\ell-\#\mesh_0\lapprox A(\theta,s)\sum_{j=0}^{\ell-1}\rho_j^{-{1}/{s}}.
\]
In view of Convergence Theorem \ref{thm:ConvConf}, we have for a factor $\Cest>0$ and a contractive factor $\alpha\in(0,1)$
\[
e_\ell^2+\Cest\eta_\ell^2\leq\alpha^{2(\ell-j)}\left(e_j^2+\Cest\eta^2_j\right),\quad j=1,..,\ell-1,
\]
holding for any iteration $\ell\ge0$.
At each intermediate step, the Efficiency Estimate \eqref{eq:result:lem:EstimatorEfficiency} makes $e_j^2+\gamma\eta^2_j\leq\textstyle\left(1+{\Cest}/{C_\rm{eff}}\right)\rho_j^2$ so we may write
\begin{equation}
\label{eq:rhoj}
\rho_j^{-\frac{1}{s}}\leq\alpha^{\frac{\ell-j}{s}}\textstyle\left(1+\frac{\Cest}{C_\rm{eff}}\right)^{\frac{1}{2s}}\left(e_\ell^2+\Cest\eta_\ell^2\right)^{-\frac{1}{2s}}.
\end{equation}
We sum \eqref{eq:rhoj} over $j=0:\ell-1$ and we recover the total-error from the quasi-error using estimator domination over oscillation,
\[
\sum_{j=0}^{\ell-1}\rho_j^{-\frac{1}{s}}\leq\sum_{j=0}^{\ell-1}\alpha^{\frac{\ell-j}{s}}\textstyle\left(1+\frac{\Cest}{\Ceff}\right)^{\frac{1}{2s}}\left(e_\ell^2+\Cest\osc_\ell^2\right)^{-\frac{1}{2s}}.
\]
We obtain
\[
\#\mesh_\ell-\#\mesh_0\lapprox M(\theta,s)\left(e_\ell^2+\Cest\osc_\ell^2\right)^{-\frac{1}{2s}}\sum_{j=1}^{\ell}\alpha^{\frac{j}{s}}
\]
where $M(\theta,s)= A(\theta,s)\left(1+\frac{\Cest}{\Ceff}\right)^{\frac{1}{2s}}$ and $\sum_{j=1}^{\ell}\alpha^{\frac{j}{s}}\leq\alpha^{1/s}(1-\alpha^{1/s})^{-1}=:S(\theta,s)$ for any $\ell\ge1$.
\[
\#\mesh_\ell-\#\mesh_0\lapprox S(\theta,s)M(\theta,s)\rho_\ell(\u,\U_\ell,f)^{-\frac{1}{s}}
\]
From Remark \ref{rem:ConvergenceAlpha}
\begin{equation}
\frac{\alpha^{1/s}}{1-\alpha^{1/s}}\leq
\end{equation}
\end{proof}
\section{Characterization of approximation classes}
In this section we characterize the approximation classes of the previous section. Namely, we will express $\cal{A}^s$, $\cal{O}^s$ and $\bb{A}^s$ in terms of Besov smoothness spaces. Let $m\ge1$ be an integer and $h>0$, we define the $m$-th order forward difference operator $\Lap_h^m$ recursively via
\begin{equation}
\Lap_h^m:=\Lap_h[\Lap_h^{m-1}],\quad\Lap_h=T_n-I,\quad T_hf(t)=f(t+h).
\end{equation}
For $G\subset\bb{R}^d$ convex with $\rm{diam}\,G=1$, we defind the Besov space via modulus the of smoothness 
\begin{equation}
\omega_m(f,t)_p:=\sup_{h\leq t}\norm{\Delta_h^mf}{L^p(G_{mh})},\quad G_{mh}=\left\{x\in G:[x+mh]^d\subset G\right\},
\end{equation}
and $\omega_m(f,t)_p:=0$ for values $t$ such that $x+mh\not\in G$.
Note that if $f\in\bb{P}_{m-1}(G)$ then $\omega_m(f,t)_p=0$, moreover if $\omega_m(f,t)_p=o(t^m)$ then necessarily we have $f\in\bb{P}_{m-1}(G)$.
For values $\alpha>0$, $0<q,p\leq\infty$ we characterize Besov spaces $\scr{B}_{q,p;m}^\alpha:=\scr{B}_{q;m}^\alpha(L^p(G))$ in terms of $\omega_m(f,t)_p$:
\begin{equation}
\scr{B}_{q,p;m}^\alpha(G)=\{f\in L^p(G):\semi{v}{\scr{B}_{q,p;m}^\alpha(G)}<\infty\}
\end{equation}
where the semi-norm reads
\begin{equation}
\semi{f}{\scr{B}_{q,p;m}^\alpha(G)}=\norm{t\mapsto t^{-\alpha-1/q}\omega_m(f,t)_p}{L^{q}(0,\infty)}
\end{equation}
Note that if $m\ge\alpha-\max\{0,1/q-1\}$ then different choices of $m$ with result in quasi-norms $\semi{\cdot}{\scr{B}^\alpha_{q,p;m}}$ that are equivalent to each other. On the other hand if $m<\alpha-\max\{0,1/q-1\}$ then the Besov space $\scr{B}_{q,p;m}^\alpha$ is a polynomial space of degree $m-1$. 
We will need some tools for the following analysis. Let $G\subset\bb{R}^2$
We will make use of the Whitney-type estimate: $0<p\leq\infty$
\begin{equation}
\inf_{\pi\in\bb{P}_{r}}\norm{f-\pi}{L^p(G)}\lapprox\omega_{r+1}(f,\rm{diam}\,G)_p\quad\forall f\in L^p(G)
\end{equation}
We have the $\omega_{r+1}(f,\rm{diam}\,G)_p\lapprox\semi{f}{\scr{B}_{p,p}^r}$ and whenever $\alpha<r$ we also have the continuous embedding $\scr{B}_{p,p}^r\hookrightarrow\scr{B}^\alpha_{p,p}$. We have
\begin{equation}\label{eq:Whitney}
\inf_{\pi\in\bb{P}_{r}}\norm{f-\pi}{L^p(G)}\lapprox\semi{f}{\scr{B}_{p,p}^\alpha}\quad\forall f\in \scr{B}_{p,p}^\alpha
\end{equation}
Let $H^0:=L^2$ and for $\beta>0$ let $H^\beta:=W^\beta_2$. We have the embedding for $\alpha>0$ and $0<p\leq\infty$
\begin{equation}\label{eq:Embeding}
\scr{B}_{p,p}^\alpha(G)\hookrightarrow H^\beta(G)
\quad\text{if}\quad\alpha-\beta>\textstyle2\left(\frac{1}{p}-\frac{1}{2}\right).
\end{equation}
Following result is due to Binev \cite{binev2002approximation} which we include for completeness.
\begin{lemma}\label{lem:AdaptiveProc}
Let $\v\in\scr{B}_{p,p}^\alpha(\Omega)$ for  $\alpha\ge0$, $0<p<\infty$ and let $\delta>0$.
\begin{equation}
e(\tau,\mesh)=|\face|^{\delta}\semi{v}{\scr{B}_{p,p}^\alpha(\omega)},\quad\omega=\tau\ \text{or}\ \omega_\tau.
\end{equation}
Given any $\eps>0$, the adaptive procedure
\begin {algorithmic}
\State $\scr{M}_\ell\leftarrow\{\face\in \mesh_\ell:e(\face,\mesh_\ell)>\eps\}$
\While {$\scr{M}_\ell\neq\emptyset$}
\State $\mesh_{\ell+1}\leftarrow\textbf{REFINE}(\mesh_\ell,\scr{M}_\ell)$
\State $\scr{M}_{\ell+1}\leftarrow\{\face\in\mesh_{\ell+1}:e(\face,\mesh_{\ell+1})>\eps\}$
\EndWhile.
\end{algorithmic}
we will terminates in finite steps and produces an admissible partition $\mesh\in\scr{P}$ for which
\begin{equation}\label{eq:lem:AdaptiveProc}
\sum_{\face\in\mesh}e(\face,\mesh)^2\lapprox\#\mesh\eps^2
\quad\text{and}\quad
\#\mesh-\#\mesh_0\lapprox\semi{v}{\scr{B}_{p,p}^{\alpha}(\Omega)}^{p/(1+\delta p)}\eps^{-p/(1+\delta p)}
\end{equation}
\end{lemma}
\begin{proof}
With each refinement step, foe error quantities $e(\tau_\ast,\mesh_\ell)$ exceeding $\eps>0$, $|\tau|$ will reduce by a factor $1/4$ and $e(\rm{child}(\tau_\ast),\mesh_{\ell+1})\leq 4^{-\delta}e(\tau_\ast,\mesh_{\ell+1})$. We will have $\scr{M}_\ell=\emptyset$ after a finite number of steps $L$; set $\mesh:=\mesh_L$ we obtain the first relation in \eqref{eq:lem:AdaptiveProc}.
We estimate the cardinality of the resulting partition $\mesh$. Let $R_\ell\subset\mesh_\ell$ be the set of refined cells and put $\cal{R}=\cup_{\ell=0}^LR_\ell$. Let $\Gamma_j=\{\tau\in\cal{R}:2^{-j-1}\leq|\tau|\leq2^{-j}\}$ and let $m_j=\#\Lambda_j$.
First of all, there can be at most $2^{j+1}|\Omega|$ disjoint $\tau$ of size $>2^{-j-1}$ which makes $m_j\leq2^{j+1}|\Omega|$ which gives us one upper bound on $m_j$.
We obtain a second upper bound in the following manner. Let $\tau\in\Gamma_j$, then
\[
e(\tau,P)=|\tau|^{\delta}\semi{v}{\scr{B}_{p,p}^{\alpha}(\omega)}<2^{-j\delta}\semi{v}{\scr{B}_{p,p}^{\alpha}(\omega)}
\]
and
\[
m_j\eps^p<\sum_{\tau\in\Gamma_j}e(\tau,\mesh)^p<2^{-jp\delta}\sum_{\tau\in\Gamma_j}\semi{v}{\scr{B}_{p,p}^{\alpha}(\omega_\tau)}^p
\lapprox2^{-jp\delta}\semi{v}{\scr{B}_{p,p}^{\alpha}(\Omega)}^p
\]
by shape-regularity. We therefore obtain $m_j\lapprox 2^{-jp\delta}\semi{v}{\scr{B}_{p,p}^{\alpha}(\Omega)}^p\eps^{-p}$.
Let $j_0$ be the smallest integer for which $|\Omega|<2^{j_0}$. Then if $\scr{M}:=\cup_{\ell=0}^L\scr{M}_\ell$
\[
\#\scr{M}\leq\sum_{j=-j_0}^\infty\#m_j
\lapprox\sum_{j=-j_0}^\infty\min\{2^{j}|\Omega|,2^{-jp\delta}\semi{v}{\scr{B}_{p,p}^{\alpha}(\Omega)}^p\eps^{-p}\}
\]
If $k$ is biggest integer for which $2^k|\Omega|\leq2^{-kp\delta}\semi{v}{\scr{B}_{p,p}^{\alpha}(\Omega)}^p\eps^{-p}$, then
\[
\begin{split}
\sum_{j=-j_0}^\infty\min\{2^{j}|\Omega|,2^{-jp\delta}\semi{v}{\scr{B}_{p,p}^{\alpha}(\Omega)}^p\eps^{-p}\}
&=
|\Omega|\sum_{j=-j_0}^k2^j+\semi{v}{\scr{B}_{p,p}^{\alpha}(\Omega)}^p\eps^{-p}\sum_{j=k+1}^\infty2^{-jp\delta}\\
\end{split}
\]
Observe that $\sum_{j=-j_0}^k2^j\lapprox 2^k$, $\sum_{j=k+1}^\infty2^{-jp\delta}\lapprox 2^{-kp\delta}$ and $2^{k(1+p\delta)}\leq|\Omega|^{-1}\semi{v}{\scr{B}_{p,p}^{\alpha}(\Omega)}^p\eps^{-p}$ which makes
\[
\begin{split}
\#\mesh-\#\mesh_0\lapprox\#\scr{M}&\lapprox 2^{-kp\delta}\semi{v}{\scr{B}_{p,p}^{\alpha}(\Omega)}^p\eps^{-p}
\leq\left(|\Omega|^\delta\semi{v}{\scr{B}_{p,p}^{\alpha}(\Omega)}\eps^{-1}\right)^{p/(1+\delta p)}
\end{split}
\]
where we invoked \eqref{eq:MarkingComplexity}.
\end{proof}
\begin{theorem}\label{thm:PreDirect}
We have $\scr{B}^{2+\alpha}_{p,p}(\Omega):=\scr{B}^{2+\alpha}_{p,p;r+1}(\Omega)\subset\cal{A}^s$ with $s=\frac{\alpha}{2}$ for values $\alpha< r-1+\max\{0,1/p-1\}$ and $\frac{\alpha}{2}\ge\frac{1}{p}-\frac{1}{2}$ and $0<p<\infty$. 
\end{theorem}
\begin{proof}
Let $\pi\in\bb{P}_r(\omega_\tau)$.
\begin{equation}
\begin{split}
\semi{v-\Ip v}{H^2(\face)}&\leq\semi{v-\pi}{H^2(\face)}+\semi{\Ip(\pi-v)}{H^2(\face)}\\
&\leq\semi{v-\pi}{H^2(\face)}+\cshape\semi{\pi-v}{H^2(\omega_\face)}
\lapprox \cshape\semi{\v-\pi}{H^2(\omega_\tau)}.
\end{split}
\end{equation}
Let $\omega_\face=T(G)$ and $\hat{v}=v\circ T$. For $\alpha< r-1+\max\{0,1/p-1\}$ we have nontrivial Besov spaces $\scr{B}_{p,p}^{2+\alpha}(G)$ and if $\frac{1}{p}\leq\frac{\alpha+1}{2}$ we have the continuous embedding $\scr{B}_{p,p}^{2+\alpha}(G)\hookrightarrow H^2(G)$. Together with the facts $\semi{\hat{\v}}{\scr{B}_{p,p}^{2+\alpha}(G)}\approx h_\face^{2+\alpha-2/p}\semi{\v}{\scr{B}_{p,p}^{2+\alpha}(\omega_\face)}$ and $\semi{\hat{\pi}}{\scr{B}^{2+\alpha}_{p,p}(G)}=0$ we arrive at
\begin{equation}
h_\face\semi{v-\pi}{H^2(\omega_\face)}\approx\semi{\hat{v}-\hat{\pi}}{H^2(G)}\lapprox\norm{\hat{v}-\hat{\pi}}{L^p(G)}+\semi{\hat{v}}{\scr{B}^{2+\alpha}_{p,p}(G)}
\end{equation}
Invoking \eqref{eq:Whitney},
\begin{equation}
\inf_{\pi\in\bb{P}_r(\omega_\tau)}h_\face\semi{v-\pi}{H^2(\omega_\face)}\lapprox\semi{\hat{v}}{\scr{B}^{2+\alpha}_{p,p}(G)}
\approx h_\tau^{2+\alpha-2/p}\semi{\v}{\scr{B}_{p,p}^{2+\alpha}(\omega_\face)}
\end{equation}
we obtain 
\begin{equation}
\inf_{\pi\in\bb{P}_r(\omega_\tau)}\semi{v-\pi}{H^2(\omega_\face)}
\lapprox h_\face^{\alpha+1-2/p}\semi{v}{\scr{B}_{p,p}^{2+\alpha}(\omega_\tau)}
\approx|\face|^\delta\semi{v}{\scr{B}_{p,p}^{2+\alpha}(\omega_\face)}
\end{equation}
with $\delta:=\frac{\alpha+1}{2}-\frac{1}{p}>0$. 
We have the local estimate
\begin{equation}
\forall\face\in\mesh,\quad \semi{v-\Ip v}{H^2(\face)}
\lapprox c_\rm{shape}|\face|^{\delta}\semi{v}{\scr{B}_{p,p}^{2+\alpha}(\omega_\tau)}
\end{equation}
and
\begin{equation}
\begin{split}
\enorm{\v-\Ip\v}{}^2&=\sum_{\tau\in\mesh}\semi{\v-\Ip\v}{H^2(\tau)}^2
\lapprox\sum_{\face\in\mesh}e(\face,\mesh)^2
\end{split}
\end{equation}
we have in view of Lemma \ref{lem:AdaptiveProc} with $\omega=\omega_\tau$, there exists an admissible mesh $\mesh\in\scr{P}$ such that
\begin{equation}\label{eq:thm:PreDirect}
\enorm{v-\Ip v}{}^2\lapprox\#\mesh\eps^2
\quad\text{with}\quad\#\mesh-\#\mesh_0\lapprox \semi{v}{\scr{B}_{p,p}^{2+\alpha}}^{{p/(1+\delta p)}}\eps^{-p/(1+\delta p)}
\end{equation}
Noting that $p\delta=p(\alpha+1)/2-1$ so $p/(1+\delta p)=2/(\alpha+1)$
let $N=\#\mesh$ and let $\eps=N^{-(\alpha+1)/2}\semi{v}{\scr{B}_{p,p}^{2+\alpha}(\Omega)}$ then
\[
\enorm{v-\Ip v}{}\lapprox \semi{v}{\scr{B}_{p,p}^{2+\alpha}(\Omega)}N^{-\alpha/2}
\quad\text{and}\quad
\#\mesh-\#\mesh_0\lapprox N.
\]
Let $s=\frac{\alpha}{2}$ then
\[
\semi{\v}{\cal{A}^s}=\sup_{N>0}N^s\inf_{\mesh\in\scr{P}_N}\inf_{\V\in\Xp}\enorm{\u-\V}{}\leq \sup_{N>0}N^s\enorm{\v-\Ip\v}{}\lapprox\semi{\v}{\scr{B}_{p,p}^{2+\alpha}(\Omega)}<\infty.
\]
\end{proof}

\begin{theorem}\label{thm:PreDirectOscilation}
We have $\scr{B}^{\alpha}_{p,p}(\Omega):=\scr{B}^{\alpha}_{p,p;r-3}(\Omega)\subset\cal{O}^s$ with $s=\frac{\alpha+1}{2}$ for values $\alpha<r-3+\max\{0,1/p-1\}$ and $\frac{\alpha}{2}\ge\frac{1}{p}-\frac{1}{2}$, $0<p<\infty$. 
\end{theorem}
\begin{proof}
Let $\pi\in\bb{P}_{r-4}(\omega_\tau)$.
\begin{equation}	
\begin{split}
\norm{f-\Pi f}{L^2(\face)}&\leq\norm{f-\pi}{L^2(\face)}+\norm{\Pi(\pi-f)}{L^2(\face)}\\
&\leq\norm{f-\pi}{L^2(\face)}+\norm{\pi-f}{H^2(\face)}
\lapprox\norm{f-\pi}{L^2(\tau)}.
\end{split}
\end{equation}
Let $\tau=T(G)$ and $\hat{f}=f\circ T$. For $\alpha<r-3+\max\{0,1/p-1\}$ we have nontrivial Besov spaces $\scr{B}_{p,p}^{\alpha}(G)$ and if $\frac{1}{p}\leq\frac{\alpha+1}{2}$ we have the continuous embedding $\scr{B}_{p,p}^{\alpha}(G)\hookrightarrow L^2(G)$. Together with the facts $\semi{\hat{f}}{\scr{B}_{p,p}^{\alpha}(G)}=h_\face^{\alpha-2/p}\semi{f}{\scr{B}_{p,p}^{\alpha}(\face)}$ and $\semi{\hat{\pi}}{\scr{B}^{\alpha}_{p,p}(G)}=0$ we arrive at
\begin{equation}
h^{-1}_\face\norm{f-\pi}{L^2(\face)}=\norm{\hat{f}-\hat{\pi}}{L^2(G)}\lapprox\norm{\hat{f}-\hat{\pi}}{L^p(G)}+\semi{\hat{f}}{\scr{B}^{\alpha}_{p,p}(G)}
\end{equation}
Invoking \eqref{eq:Whitney},
\begin{equation}
\inf_{\pi\in\bb{P}_{r-4}(\tau)}h^{-1}_\face\norm{f-\pi}{L^2(\face)}\lapprox\semi{\hat{f}}{\scr{B}^{\alpha}_{p,p}(G)}
= h_\tau^{\alpha-2/p}\semi{f}{\scr{B}_{p,p}^{\alpha}(\face)}
\end{equation}
we obtain 
\begin{equation}
\inf_{\pi\in\bb{P}_r(\tau)}\norm{f-\pi}{L^2(\face)}
\lapprox h_\face^{\alpha+1-2/p}\semi{f}{\scr{B}_{p,p}^{\alpha}(\tau)}
\end{equation}
We have 
\begin{equation}
\begin{split}
\osc_\mesh(f)&\approx\sum_{\tau\in\mesh}h_\tau^2\norm{f-\Pi f}{L^2(\tau)}\\
&\lapprox\sum_{\tau\in\mesh}h_\face^{\alpha+3-2/p}\semi{f}{\scr{B}_{p,p}^{\alpha}(\tau)}
\approx\sum_{\tau\in\mesh}|\tau|^\delta\semi{f}{\scr{B}_{p,p}^{\alpha}(\tau)}
\end{split}
\end{equation}
with $\delta:=\frac{\alpha+3}{2}-\frac{1}{p}$; let $e(\tau,\mesh)=|\tau|^\delta\semi{f}{\scr{B}_{p,p}^{\alpha}(\tau)}$.
we have in view of Lemma \ref{lem:AdaptiveProc} with $\omega=\tau$, there exists an admissible mesh $\mesh\in\scr{P}$ such that
\begin{equation}\label{eq:thm:PreDirect}
\osc_\mesh^2(f)\lapprox\#\mesh\eps^2
\quad\text{with}\quad\#\mesh-\#\mesh_0\lapprox \semi{v}{\scr{B}_{p,p}^{\alpha}}^{p/(1+\delta p)}\eps^{-p/(1+\delta p)}
\end{equation}
Noting that $p\delta=p(\alpha+3)/2-1$ so $p/(1+\delta p)=2/(\alpha+3)$,
let $N=\#\mesh$ and let $\eps=N^{-(\alpha+3)/2}\semi{v}{\scr{B}_{p,p}^{\alpha}(\Omega)}$ then
\[
\osc_\mesh(f)\lapprox \semi{f}{\scr{B}_{p,p}^\alpha(\Omega)}N^{-(\alpha+1)/2}
\quad\text{and}\quad
\#\mesh-\#\mesh_0\lapprox N.
\]
Let $s=\frac{\alpha+1}{2}$ then $\semi{f}{\cal{O}^s}<\infty$.
\end{proof}
The previous two results in combination with Lemma \ref{lem:EquivClasses} yields a one-sided characterization of the AFEM approximation class $\bb{A}^s$:
\begin{corollary}[One-sided characterization for $\bb{A}^s$]
Let $\u$ be the weak solution to \eqref{eq:cwp}. If $\u\in(\scr{B}_{p,p;r+1}^{2s+2}(\Omega)\cap H^2_0(\Omega))$ with $1/p-1/2\leq2s<r-1+\max\{0,1/p-1\}$ for some $0<p<\infty$ and $\cal{L}\u\in(\scr{B}^{2s-1}_{q,q;r-3}(\Omega)\cap L^2(\Omega))$ with $1/q-1/2\leq2s<r-3+\max\{0,1/q-1\}$ for some $0<q<\infty$, then $u\in\bb{A}^s$.
\end{corollary}
\bibliography{publications}
\bibliographystyle{siam}

\end{document}